\documentclass{amsart}
\pdfoutput=1
\usepackage[margin=1in]{geometry}
\usepackage{graphicx}
\usepackage{url}
\newtheorem{thm}{Theorem}[section]
\newtheorem{prop}[thm]{Proposition}
\newtheorem{cor}[thm]{Corollary}

\newtheorem{lemma}[thm]{Lemma}

\newtheorem{preremark}[thm]{Remark}
\newenvironment{remark}{\begin{preremark}\rm}{\medskip \end{preremark}}
\numberwithin{equation}{section}

\newcommand{\norm}[1]{\left\Vert#1\right\Vert}

\newcommand{\set}[1]{\left\{#1\right\}}
\newcommand{\R}{\mathbb R}
\newcommand{\eps}{\varepsilon}

\newcommand{\grad} {\nabla}

\newcommand{\dd} {\; \mathrm{d}}

\newcommand{\one} {\chi}

\begin{document}

\title{A new regularization mechanism for the Boltzmann equation without cut-off}

\author{Luis Silvestre}
\address{Department of Mathematics, The University of Chicago}
\email{luis@math.uchicago.edu}

\thanks{Luis Silvestre was partially supported by NSF grants DMS-1254332 and DMS-1065979.}

%\date{\today}

\begin{abstract}
We apply recent results on regularity for general integro-differential equations to derive a priori estimates in H\"older spaces for the space homogeneous Boltzmann equation in the non cut-off case. We also show an a priori estimate in $L^\infty$ which applies in the space inhomogeneous case as well, provided that the macroscopic quantities remain bounded.
\end{abstract}

\maketitle

\section{Introduction}
The purpose of this article is to analyse the implications of recent regularity results about general integro-differential equations \cite{schwab} for the non cut-off homogeneous Boltzmann equation. There is a vast literature about the regularity of the solutions of the homogeneous non cut-off Boltzmann equation. Essentially all these results are based on coercivity bounds of the general form
\begin{equation} \label{e:subellipticity}  \int_{\R^N} Q(g,f) f \dd v \geq c \|f\|_{H^{\nu/2}_\ast}^2 - \text{(lower order correction)}.
\end{equation}
Here $Q$ is the Boltzmann collision operator. The constant $c$ should only depend on the mass, energy and entropy of $g$. The main point of this paper is that it provides a priori regularity estimates for the non cut-off homogeneous Boltzmann equation without using any coercivity estimate like \eqref{e:subellipticity}, even indirectly.

Until the late 90's, most of the work on the Boltzmann equation was concentrated on the simpler cut-off case. After the entropy dissipation estimates for the non cut-off case were obtained in \cite{alexandre2000entropy}, the study of the non cut-off case exploded and many results were produced. The main estimate in \cite{alexandre2000entropy} was essentially an estimate of the form \eqref{e:subellipticity}. It was later used in \cite{desvillettes2005smoothness} to prove that, under some conditions, the solutions of the homogeneous non cut-off Boltmann equation belong to the Schwartz space. Many improvements of these results were obtained during the following 10 years, which we review below. The source of regularization in all these works is an estimate like \eqref{e:subellipticity}, which is proved using Fourier analysis. The great novelty of this article is that the regularization effect comes from a completely different method. The regularity theory for general integro-differential equations has developed in parallel in the last fifteen years motivated mainly by applications to Levy processes. This is the first paper where those results are applied to the Boltzmann equation. %We stress the novelty of the method.

The Boltzmann equation consists in finding a non negative function $f(t,x,v)$ solving
\begin{equation} \label{e:Boltzmann-inhomogeneous}  
\begin{aligned}
f_t + v \cdot \grad_x f &= Q(f,f),\\
f(t,x,v) &= f_0(x,v).
\end{aligned}
\end{equation}

The right hand side is written in terms of the bilinear form $Q$, which has the following form
\begin{equation} \label{e:Boltzmann-Q}
Q(f,g) = \int_{\R^N} \int_{\partial B_1} \left( f(v'_\ast) g(v') - f(v_\ast) g(v) \right) B\left( r,\theta \right) \dd \sigma \dd v_\ast.
\end{equation}
Here
\begin{align}
r &= |v_\ast-v| = |v_\ast'-v'|, \label{e:var-r}\\
\cos \theta &= \sigma \cdot \frac{v_\ast-v}{|v_\ast-v|}, \label{e:var-theta}\\
v' &= \frac{v+v_\ast} 2 + \frac{r} 2 \sigma, \label{e:var-vprime}\\
v_\ast' &= \frac{v+v_\ast} 2 - \frac{r} 2 \sigma.\label{e:var-vstarprime}
\end{align}

%Note that for any Gaussian function $f$, the integrand in \eqref{e:Boltzmann-Q} is identically zero. Thus, they are all stationary solutions. These stationary solutions are called \emph{Maxwellians}.

There are different modelling choices for $B$. In this paper we will consider the usual cross sections $B$ of the form
\begin{equation} \label{e:B-assumption} B(r,\theta) = r^\gamma b(\cos \theta), \text{ where } \gamma > -N \text{ and } b(\cos \theta) \approx |\theta|^{-(N-1) - \nu} \text{ for } \theta \in (-\pi,\pi).
\end{equation}
We write $b(\cos \theta)$ instead of $b(\theta)$ to stress that it is an even function in $\theta$ which is $2\pi$-periodic. When we write $b(\cos \theta) \approx |\theta|^{-(N-1) - \nu}$, we mean that there is some constant $C$ such that $C^{-1} |\theta|^{-(N-1) - \nu} \leq b(\theta) \leq C|\theta|^{-(N-1) - \nu}$. The lower bound should be only necessary in a neighbourhood of $\theta=0$. For simplicity, we assume that the lower bound holds for all values of $\theta \in (-\pi,\pi)$. Whenever we say in this article that an estimate depends on the cross section $B$, we mean that it depends only on $\gamma$, $\nu$ and the constant $C$ above. No other finer condition on $B$ affects the estimates.

The physically relevant model derived from inverse-power potentials gives
\[ \gamma = \frac{s-(2N-1)}{s-1}, \qquad \nu = \frac 2 {s-1}. \]
The parameter $s$ is supposed to be larger than $2$. Note that we have $\nu \in (0,2)$ and $\gamma = 1 - (N-1) \nu$. In this article we will consider arbitrary pairs $\gamma > -N$ and $\nu \in (0,2)$.

%As $s \to 2$, the integral in $Q(f,f)$ becomes singular and the Boltzmann equation formally converges to the Landau equation. The case of $\gamma < 0$ is called soft potential, $\gamma = 0$ is the Maxwellian potential, and $\gamma>0$ is hard potential. In this article, we will consider $\gamma > -N$ and $\nu \in (0,2)$.

In these models $B(r,\theta)$ is never integrable for $\sigma \in \partial B_1$. Using a modified cross section $B$ which is integrable is known as Grad's cut-off assumption. It is essential for the methods in this article that we do \textbf{not} make the cut-off assumption.

The space homogeneous problem is the case where $f$ is independent of $x$ and the equation reduces to
\begin{equation} \label{e:space-homogeneous-boltzmann}  f_t = Q(f,f). 
\end{equation}

Throughout this paper, we assume that our solutions are classical and derive a priori estimates. Using the vanishing viscosity method, one should be able to prove that there is a weak solution which satisfies these estimates. This vanishing viscosity procedure is compatible with the estimates in \cite{schwab}, although this is not explicitly worked out in that paper.

The main application of the results from \cite{schwab} gives us a H\"older estimate for the linear Boltzmann equation. We state the result here, but a more precise description is given as Theorem \ref{t:Calpha-linearized}.

\begin{thm} \label{t:calpha-linearized-intro}
Assume $B$ has the form \eqref{e:B-assumption}. Let $f:[0,T] \times \R^N$ be a non negative function whose mass, energy and entropy stay bounded for all $t \in [0,T]$. In the case of the mass, we also require it to stay strictly positive.

Let $g$ be a solution of the linear problem
\[ g_t = Q(f,g) + h,\]
then, for any $R>0$, there exists an $\alpha>0$ and $C>0$ so that $g$ satisfies the estimate
\[ \|g\|_{C^\alpha(B_R \times [T/2,T])} \leq C \left( \|g\|_{L^\infty(\R^N \times [0,T])} + \|h\|_{L^\infty(B_{2R} \times [0,T])} \right).\]
The values of $\alpha$ and $C$ depend on $R$, $T$, the mass, energy and entropy of $f$, and the two constants $K_0$ and $\tilde K_0$ below
\begin{align*}
K_0 &:= \sup_{(t,v) \in [0,T]\times B_R} \int_{\R^N} |w|^\gamma f(t,v+w) \dd v && \text{\bf (only necessary if $\gamma < 0$)}, \\
\tilde K_0 &:= \sup_{(t,v) \in [0,T]\times B_R} \int_{\R^N} |w|^{\gamma+\nu} f(t,v+w) \dd v  && \text{\bf (only necessary if $\gamma+\nu > 2$)}
\end{align*}
\end{thm}

The constants $C$ and $\alpha$ in Theorem \ref{t:calpha-linearized-intro} depend only on upper bounds for the energy, entropy, $K_0$ and $\tilde K_0$. With respect to the mass, they depend on both upper and lower bounds.

Note that the value of $K_0$ is controlled by the mass of $f$ and $\tilde K_0$ if $\gamma \geq 0$, which makes it irrelevant in that case. Moreover, if $\gamma+\nu \leq 2$, then the value of $\tilde K_0$ is controlled by $K_0$ and the energy of $f$.

The previous theorem can be applied to $g=f$ with $h=0$. If we make $g$ equal to a derivative of $f$, then we would get lower order terms into the right hand side $h$. This might be used as a possible bootstrapping idea but will not be pursued in this article. 

An obvious application of Theorem \ref{t:calpha-linearized-intro} is that whenever the solution $f$ of \eqref{e:space-homogeneous-boltzmann} is bounded, it is also locally $C^\alpha$. Note that if $f \in L^\infty([0,T],L^\infty(\R^N) \cap L^1(\R^N))$, then the existence of the constant $K_0$ in the assumptions of Theorem \ref{t:calpha-linearized-intro} is guaranteed even for $\gamma<0$.

In order to prove Theorem \ref{t:calpha-linearized-intro}, we rewrite the Boltzmann equation as $g_t = Q_1(f,g) + Q_2(f,g)$ where $Q_1(f,g)$ is an integro-differential operator of the form studied in \cite{schwab} and $Q_2(f,g)$ is a lower order term. The application of the main result in \cite{schwab} gives us the H\"older estimate. It is interesting to point out that previous H\"older estimates for parabolic integro-differential equations (like \cite{silvestreHJ}, \cite{silvestreASNSP}, \cite{lara2014regularity}, \cite{luislecturenotes} section 3.2 or \cite{silvestreICM}) do not suffice to prove Theorem \ref{t:calpha-linearized-intro} since they would depend on higher regularity conditions for the function $f$.

We also obtain $L^\infty$ estimates to complement Theorem \ref{t:calpha-linearized-intro}. We state this theorem in terms of the space in-homogeneous equation. A more precise version is given as Theorem \ref{t:Linfty}.

\begin{thm} \label{t:Linfty-intro}
Let $f$ be a solution of the Boltzmann equation \eqref{e:Boltzmann-inhomogeneous} which is periodic in $x$. Assume that for all $x \in \R^N$ and $t \geq 0$, we have
\begin{align*}
0< M_1 \leq \int_{\R^N} f(t,x,v) \dd v &\leq M_0, \\
\int_{\R^N} |v|^2 f(t,x,v) \dd v &\leq E_0, \\
\int_{\R^N} f(t,x,v) \log f(t,x,v) \dd v &\leq H_0.
\end{align*}
Moreover, if $\gamma > 2$ or $\gamma+\nu \leq 0$, we need to make the following extra assumptions. There exists some $p > N / (N+\nu+\gamma)$, so that for $q = \max(0,1-\frac N \nu (\nu+\gamma))$ and for all $x \in \R^N$ and $t \geq 0$,
\begin{align*} 
\int_{\R^N} f(t,x,v) |v|^\gamma \dd v &\leq \tilde K_0. \qquad \text{\bf (only necessary if $\gamma>2$)}\\
\sup_{t \in [0,T]} \int_{\R^N} (1+|v|)^q f(t,x,v)^p \dd v &\leq K_0. \qquad \text{\bf (only necessary if $\gamma+\nu \leq 0$)}
\end{align*}

Then, an $L^\infty$ bound holds
\[ \|f(t,\cdot)\|_{L^\infty} \leq C(t),\]
Here $C(t) < +\infty$ for all $t>0$. It is a function which depends only on $E_0$, $M_0$, $M_1$ and $H_0$, $\tilde K_0$ (if $\gamma > 2$) and $K_0$ (if $\gamma+\nu \leq 0$).
\end{thm}

Note that if $\gamma \in [0,2]$, the value of $K_0$ is controlled by $M_0$ and $E_0$. If $\gamma<0$, the value of $K_0$ is not used in the estimate. If $\gamma+\nu > 0$, then we can take $p=1$, $q<1$ and $\tilde K_0 = C(M_0,E_0)$.

As a final result for the Boltzmann equation, Theorem \ref{t:Linfty-intro} is probably the most novel result of this paper compared with the available literature. Note that when $\gamma+\nu > 0$ and $\gamma \leq 2$, Theorem \ref{t:Linfty-intro} gives us a regularization result in $L^\infty$ for the inhomogeneous Boltzmann equation which depends only on physically relevant quantities. The initial data $f_0$ need not be bounded.

Combining Theorems \ref{t:calpha-linearized-intro} and \ref{t:Linfty-intro}, we obtain a priori estimates for solutions of the space homogeneous Boltzmann equation without cut-off. The first result concerns the H\"older continuity of the solutions of \eqref{e:space-homogeneous-boltzmann}. A more precise version is given as Corollary \ref{c:Calpha-Boltzmann}

\begin{cor} \label{c:Calpha-Boltzmann-intro}
Assume $B$ has the form \eqref{e:B-assumption}. Let $R>0$ and $f:[0,T] \times \R^N \to \R$ be a non negative solution of \eqref{e:space-homogeneous-boltzmann} with finite mass, energy and entropy.

In the case $\gamma+\nu \leq 0$, we also assume,
\[ 
\begin{aligned}
\sup_{t \in [0,T]} \int_{\R^N} (1+|v|)^q f(t,v)^p \dd v &\leq K_0, \\
\text{ for some  } p &> N / (N+\nu+\gamma) \text{ and } q = \max(0,1-\frac N \nu (\nu+\gamma)). \ \ \text{\bf (only necessary if $\gamma+\nu \leq 0$)}
\end{aligned}
\]
Then, there is an a priori estimate of the form
\[ \|f\|_{L^\infty([T/2,T]\times \R^N)} + \|f\|_{C^\alpha([T/2,T] \times B_{R/2})} \leq C.\]
where $\alpha>0$ and $C$ depend only on the mass, energy and entropy of $f_0$, $K_0$, $R$, $T$, the dimension $N$ and the cross section $B$.
\end{cor}

%The second result concerns the $C^{1,\alpha}$ regularity of the solutions to \eqref{e:space-homogeneous-boltzmann}. In this case we assume that $|\grad f_0|$ is initially bounded. A more precise version of the result is given as Corollary \ref{c:C1alpha-Boltzmann}.

%\begin{cor} \label{c:C1alpha-Boltzmann-intro}
%Assume $B$ has the form \eqref{e:B-assumption}. Let $R>0$ and $f$ be a non negative solution of \eqref{e:space-homogeneous-boltzmann} with finite mass, energy and entropy. Assume also that
%\begin{align*}
%\sup_{v \in \R^N} |\grad f(0,v)| \leq G_0, \\
%\sup_{t \in [0,T]} \int_{\R^N} (1+|v|)^q f(t,v)^p \dd v &\leq K_0, \\
%\text{ for some  } p &> N / (N+\nu+\gamma) \text{ and } q = \max(0,1-\frac N \nu (\nu+\gamma)). \ \  \text{\bf (only necessary if $\gamma+\nu \leq 0$)}
%\end{align*}
%Then, there is an a priori estimate of the form
%\[ \|\grad f\|_{L^\infty([0,T] \times \R^N)} + \|f\|_{C^{1,\alpha}([T/2,T] \times B_{R/2})} \leq C.\]
%where $\alpha>0$ and $C$ depend only on the mass, energy and entropy of $f_0$, $G_0$, $K_0$, $R$, $T$, the dimension $N$ and the cross section $B$.
%\end{cor}

Another application of the results from \cite{schwab} to the Boltzmann equation gives us a quantitative lower bound. This is a rather direct application of the weak Harnack inequality for integro-differential equations.

\begin{thm} \label{t:lowerbound-intro}
Assume $B$ is of the form \eqref{e:B-assumption}. Let $f$ be a solution to the homogeneous non-cutoff Botzmann equation \eqref{e:space-homogeneous-boltzmann}. Let $R>0$ and $T>0$. Assume that the initial data $f_0$ has finite mass, energy and entropy. If $\gamma+\nu < 0$, assume also that 
\[ \sup_{(t,v) \in (0,T) \times B_R} \int_{\R^N} |v-w|^{\gamma+\nu} f(t,w) \dd w \leq K_0.\]
Then there exists a lower bound $c$ so that
\[ f(T,v) \geq c, \ \ \text{ for all } v \in B_R. \]
This constant $c$ depends on $T$, $R$, the dimension $N$, the initial mass, energy and entropy, the cross section $B$ satisfying \eqref{e:B-assumption}, and also $K_0$ if $\gamma+\nu < 0$.
\end{thm}

This is a weaker result than the lower bound obtained by Mouhot in \cite{Mouhot-lowerbounds}, since we do not have an explicit control on how the lower bound deteriorated as $R$ increases. The result of \cite{Mouhot-lowerbounds}, in the non cut-off case, requires us to assume that the solution is in $L^\infty(W^{2,\infty})$. This follows from well known regularity estimates in many cases, but not for the full range of parameters $\gamma \in (-N,+\infty)$ and $\nu \in (0,2)$. Our result has the modest advantage that it requires lower regularity conditions for those parameters for which the smoothness of the solution is not a priori known.

\noindent \textbf{Notation. } Throughout the paper we use $a \approx b$ to denote that there is a universal constant $C$ such that $C^{-1} a \leq b \leq C a$.

We also use the names $v$, $v'$, $v_\ast$, $v'_\ast$, $\sigma$, $r$ and $\theta$ to be the quantities related by 
\eqref{e:var-r}, \eqref{e:var-theta},\eqref{e:var-vprime} and \eqref{e:var-vstarprime}. This means that we often take some of this values implicitly as functions in terms of the others.

When we say that a set $\Xi \subset \R^N$ is a cone, we mean simply that $\gamma \Xi =\Xi$ for any $\gamma >0$. That means that $\Xi$ is a union a rays coming from the origin. Any cone is characterized by its intersection to the unit sphere. Thus, if $A$ is any \textbf{arbitrary} subset of the unit sphere (i.e. $A \subset \partial B_1$) then
\[ \Xi = \{ r \sigma : \sigma \in A \text{ and } r>0\},\]
is a cone.

\subsection{Comparison with previous results}

Coercivity estimates like \eqref{e:subellipticity} were used extensively in the last 15 years to obtain results for the non cut-off Boltzmann equation. Its first appearance was probably in \cite{MR1324404} as an estimate for the Kac equation. Other early forms of this coercivity estimates appeared in \cite{MR1407542}, \cite{MR1475459} and \cite{MR1649477}. In \cite{alexandre2000entropy} they found a lower bound for the entropy dissipation which was based on an estimate of the form \eqref{e:subellipticity}. Variations of this estimate were used to obtain various regularity results for the space homogeneous Boltzmann equation without cut-off in \cite{alexandre2002boltzmann}, \cite{desvillettes2005smoothness}, \cite{alexandre2005littlewood}, \cite{MR2425608}, \cite{MR2556715}, \cite{morimoto2009regularity}, \cite{alexandre2007littlewood}, \cite{MR2525118}, \cite{MR2679746}, \cite{chen2011smoothing}, \cite{MR2959943} and \cite{MR2917172}. It was also used for other types of results, including estimates for solutions in the space in-homogeneous case in \cite{alexandre2002boltzmann}, \cite{MR2543976}, \cite{MR2677982}, \cite{MR2679369}, \cite{MR2795331}, \cite{alexandre2011boltzmann}, \cite{MR2863853}, \cite{MR2885564}, \cite{MR2629879} and \cite{MR2784329}. Coercive inequalities are obviously a tremendously successful tool in the study of the non cut-off Boltzmann equation. The following paragraph is copied from \cite{MR2556715}:
\begin{quote}
{\em ``Turning to Boltzmann equation, as usual in the general theory of Partial
Differential Equations, it is important to have (good) a priori estimates.
The natural ones for Boltzmann equation (and up to now, the only known ones at
the exception of one dimensional models) are based on symmetrization properties
of the collision operator.''}
\end{quote}
The purpose of this paper is to show a new approach to obtain a priori estimates on the regularity of the solution $f$ of the Boltzmann equation. We use the regularity theory for parabolic integro-differential equations which was recently developed (especially the result in \cite{schwab}). We stress that no estimate of the form \eqref{e:subellipticity} is used in our proof even indirectly. Usually coercive inequalities like \eqref{e:subellipticity} are proved using Fourier analysis. The methods in this article do not use Fourier analysis or even estimates in terms of Sobolev norms at all. 

The final results that we obtain are not necessarily better than the results in the literature for the space homogeneous Boltzmann equation without cut-off. This is not surprising given the maturity of the coercivity methods. We believe that the main value of this article lies in the introduction of the new link between the Boltmann equation and the recent results in integro-differential equations. We still compare the final results below.

We should also point out that a different method to obtain some regularization results was given in \cite{MR1768239} and \cite{MR2851696}. In the second article, the authors study the space homogeneous Boltzmann equation in 2D with hard potentials ($\gamma>0$) using probabilistic methods. They obtain an estimate in a Sobolev space assuming that the solution $f$ has exponential moments. The estimate is in better Sobolev spaces when $\gamma$ is large.

Our result in Corollary \ref{c:Calpha-Boltzmann-intro} gives estimates in $C^\alpha_{loc}$. This is weaker than the usual estimates in Schwartz spaces or $C^\infty$, using coercivity estimates. At the moment, it is not clear whether we can obtain higher derivative estimates using the methods of this paper. One possibility would be to apply the methods from \cite{serra2014regularity} or \cite{jin2014schauder}, but more work needs to be done. The main difficulty is that the estimate in Theorem \ref{t:calpha-linearized-intro} is local, and that makes it harder to bootstrap a higher derivative estimate. This difficulty was essentially present in the early works using coercive estimates. Indeed, the estimate in \cite{alexandre2000entropy} is in a local Sobolev space. The analysis of the proof in \cite{alexandre2000entropy} shows that the estimate is global only when $\gamma=0$. Because of this fact, the first smoothness result in \cite{desvillettes2005smoothness} assumes that 
\[ B(r,\theta) = \Psi(r) b(\cos(\theta)),\]
where $\Psi(r)$ is strictly bounded below. It took a long time to extend the smoothness result to the natural cross sections of the form \eqref{e:B-assumption} for values $\gamma \neq 0$. This was achieved in \cite{chen2011smoothing}, \cite{MR2677982} and \cite{MR2917172} obtaining coercive estimates in global weighted Sobolev spaces. At the moment, it is not clear whether the estimate in Theorem \ref{t:calpha-linearized-intro} can be made global for any pair $(\gamma,\nu)$.

Interestingly, the results available are much more satisfactory when $\gamma+\nu > 0$, as well as our Corollary \ref{c:Calpha-Boltzmann-intro}. Indeed, the result in \cite{MR2677982} holds only when $\gamma+\nu > 0$, and the result in \cite{chen2011smoothing} is unconditional only when $\gamma+\nu > 0$. For the case $\gamma+\nu \leq 0$, in \cite{chen2011smoothing} they assume that $\gamma+\nu > -1$ and $\langle v \rangle^l f(v) \in L^\infty([0,T],L^{3/2})$ for all $l>0$. They work in the 3D case, and this is a slightly stronger condition than the one we impose in Corollary \ref{c:Calpha-Boltzmann-intro} when $\gamma+\nu \leq 0$. Indeed, note that $3/2 > 3 / (3+\gamma+\nu)$ if $\gamma+\nu > -1$.

The existence of smooth solutions for the space homogeneous Boltzmann equation is not well understood when $\gamma+\nu \leq 0$. There is no unconditional regularity result by any known method. Note that in 3D, if $\gamma \to -3$ and $\nu \to 2$, the equation converges to the Landau equation which corresponds to Coulombic potentials. Thus, there is strong physical relevance to the study of this range. See Remark \ref{r:enhanceddiffusion} for some discussion of why the case $\gamma+\nu < 0$ may not be achievable by any of the current methods.

There is a physical interpretation of the sign of $\gamma+\nu$ that we now describe. Let $f_1$ and $f_2$ be two non negative smooth functions with compact support and $f(v) = f_1(v) + f_2(v-v_0)$. Assume that $f_1$ achieves its maximum at $v=0$ and $f_2(0) > 0$. One can easily see, for example using the expressions from Lemma \ref{l:expression-for-Q2} and following the proof of Lemma \ref{l:K-above}, that
\[
\lim_{|v_0| \to \infty} Q(f,f)(0) - Q(f_1,f_1)(0) = \begin{cases}
0 & \text{ if } \gamma+\nu < 0, \\
-\infty & \text{ if } \gamma+\nu > 0.
\end{cases}
\]
In the case $\gamma+\nu < 0$ the equation is more \emph{local}, which makes it harder to understand.

The $L^\infty$ bound of Theorem \ref{t:Linfty-intro} appears to be new in the space inhomogeneous case. In \cite{gamba2009upper}, they obtained a Maxwellian upper bound in the homogeneous cut-off case, assuming that the initial value is also bounded. They use maximum principle ideas, based on the discussion in section 6, chapter 2, in \cite{villani2002review}. There, the basic decomposition $Q=Q_1+Q_2$, which we describe in depth in section \ref{s:twoterms}, is outlined. Our proof of Theorem \ref{t:Linfty-intro} does not follow the same method as in \cite{gamba2009upper} but some ideas are naturally common. In this proof, we do not apply any result from \cite{schwab} either. We do follow some maximum principle intuition from integro-differential equations.

Using coercivity estimates, conditional $C^\infty$ regularity results for the space in-homogeneous equation where obtained in \cite{MR2543976}, \cite{MR2679369} and \cite{MR2885564}. Those results require the assumption that the solution $f$ is bounded in $H^5$ with respect to all variables. Theorem \ref{t:Linfty-intro} gives only $L^\infty$ regularity, but it depends on a much milder a priori assumption on the solution $f$. Indeed, if $\gamma \leq 2$ and $\gamma+\nu > 0$, then it only depends on the macroscopic quantities staying under control.

At the moment, it seems that in order to extend the result of Theorem \ref{t:calpha-linearized-intro} to the space inhomogeneous case, we would need new developments in the regularity theory of integro-differential equations.

\section{Preliminaries}

\subsection{Macroscopic quantities}

It is well known that the Boltzmann equation conserves mass and energy. Its entropy is non increasing. In the space homogeneous case, these quantities are defined by the formulas

\begin{description}
\item[Conservation of mass.] 
The mass $M$ is independent of time.
\begin{equation} \label{e:mass}  M := \int_{\R^N} f(t,v) \dd v
\end{equation}
\item[Conservation of energy.] The energy $E$ is independent of time.
\begin{equation} \label{e:energy}  E = \int_{\R^N} f(t,v) |v|^2 \dd v 
\end{equation}
\item[Entropy.] The entropy $H(t)$ is monotone decreasing in time.
\begin{equation} \label{e:entropy}  H(t) = \int_{\R^N} f(t,v) \log f(t,v) \dd v 
\end{equation}
\end{description}

%Weak solutions to the Boltzmann equations satisfy the same conditions except that the energy is not known to be conserved, but instead it is just non increasing. 

\subsection{Estimates for general integro-differential equations}

\label{s:pie}

In this section we explain the relevant results from \cite{schwab}. We adapt the notation only slightly in order to match, to some extent, the notation which will be used below related to the Boltzmann equation.

An integro-differential equation involves a non negative kernel
 $K: [t_0,t_1] \times B_R \times \R^N \to \R$ and a function $u$ (the \emph{solution}) which satisfies the equation
\[ u_t(t,v) - \left( \int_{\R^N} \left( u(t,v') - u(t,v) \right) K(t,v,v') \dd v' \right) = h(t,v).\]
Here, the function $h$ is a right hand side.

Following \cite{schwab}, we make the following three assumptions on $K$. Here $\lambda$, $\Lambda$ and $\mu$ are three arbitrary positive constants (independent of $v$, $t$, $v'$, etc...).

We assume $K$ has the following symmetry property. For all $v \in B_R$, $t \in [t_0,t_1]$ and $w \in \R^N$, 
\begin{equation} \label{e:K-symmetric}  K(t,v,v+w) = K(t,v,v-w).
\end{equation}

We assume the following upper bound on average. There exists some constant $\Lambda>0$, so that for all $(t,v) \in [t_0,t_1] \times B_R$,
\begin{equation} \label{e:K-bound-above}  \int_{B_{2r}(v) \setminus B_r(v)} K(t,v,v') \dd v' \leq \Lambda r^{-\nu}.
\end{equation}

We assume that, at every point, $K$ is bounded below in a cone of directions. More precisely, that for every point $v \in B_R$ and $t \in [t_0,t_1]$, there exists a set $A = A(t,v) \subset \partial B_1$ such that its measure on the sphere is bounded below, $|A| \geq \mu > 0$, and we have
\begin{equation} \label{e:K-bound-below}  K(t,v,v+r\sigma) \geq \lambda r^{-N-\nu}, \ \text{ for all } r>0 \text{ and } \sigma \in A.
\end{equation}

There is a small change of notation between our integro-differential equations in this paper and in \cite{schwab}. What here we write $K(t,v,v')$ corresponds to $K(t,v,v'-v)$ in \cite{schwab}. The reason for this change in notation is to write everything in terms of $v$ and $v'$ which are commonly used in Boltzmann literature, instead of writing everything in terms of $v$ and $v'-v$.

The symmetry assumption \eqref{e:K-symmetric} is not made in \cite{schwab}. It simplifies some assumptions and computations. In particular, the assumption (A4) in \cite{schwab} is irrelevant. Since the kernels we obtain from the Boltzmann equation are symmetric, we assume symmetry here.

The assumption \eqref{e:K-bound-above} corresponds to the assumption (A2) in \cite{schwab}. The assumption \eqref{e:K-bound-below} corresponds to the assumption (A3) in \cite{schwab}. Note that we will always have $A = -A$ because of the symmetry assumption.

The following theorem is sometimes referred as \emph{weak-Harnack} inequality. It is a weaker version of Theorem 6.1 in \cite{schwab}.

\begin{thm} \label{t:weak-harnack}
Assume that $u \geq 0$ in $[0,T] \times \R^N$ and satisfies
\[ u_t(t,v) - \left( \int_{\R^N} \left( u(t,v') - u(t,v) \right) K(t,v,v') \dd v' \right) \geq 0,\]
for all $t \in [0,T]$ and $v \in B_R$. Assume also that $K$ satisfies \eqref{e:K-symmetric}, \eqref{e:K-bound-above} and \eqref{e:K-bound-below} and moreover, there is an $\ell>0$ and $m>0$ such that
\[ |\{(t,v) \in [0,T/2] \times B_{R/4} : u(t,v)>\ell\}| \geq m,\]
then
\[ u(t,v) \geq c\ell,\]
for all $t \in [T/2,T]$ and $v \in B_{R/4}$. Here $c>0$ is a constant depending on $\nu$, $\lambda$, $\Lambda$, $\mu$, $m$ and $N$.
\end{thm}

The following theorem gives us H\"older estimates. It is a version of Theorems 7.1 and 7.2 in \cite{schwab}.

\begin{thm} \label{t:Calpha-general}
Assume that $u$ is a bounded function in $[0,T] \times \R^N$ and satisfies
\[ u_t(t,v) - \left( \int_{\R^N} \left( u(t,v') - u(t,v) \right) K(t,v,v') \dd v' \right) = h(t,v),\]
for all $t \in [0,T]$ and $v \in B_R$. Assume also that $K$ satisfies \eqref{e:K-symmetric}, \eqref{e:K-bound-above} and \eqref{e:K-bound-below}. We have
\[ \|u\|_{C^\alpha([T/2,T] \times B_{R/2})} \leq C \left( \|u\|_{L^\infty([0,T] \times \R^N)} + \|h\|_{L^\infty([0,T] \times B_R)} \right).\]
\end{thm} 

In \cite{schwab}, the solutions are understood in the viscosity sense and the equations are restated in terms of maximal operators. For the purpose of this paper we assume that we have classical solutions and we focus on the a priori estimates. The concept of viscosity solution used in \cite{schwab} is not necessarily compatible with the usual concept of weak solution for the Boltzmann equation.

\section{The Boltzmann collision operator as the sum of two terms}
\label{s:twoterms}

In this section we rewrite the bilinear form $Q$ as a sum of two terms. The first term, which we call $Q_1$, has the form of the integro-differential operators matching those described in section \ref{s:pie}. The second term, which we call $Q_2$ is, in some sense, a lower order term.

Note that the value of $Q(f,f)$ is not affected if we replace $B(r, \theta)$ for another function $\tilde B$ such that
\[ B(r,\theta) + B(r,\theta+\pi) = \tilde B(r,\theta) + \tilde B(r,\theta+\pi).\]
There are a few canonical choices that we can make here. It would be natural (and it is somewhat common) to have all non zero values on one side by imposing $B(r,\theta)=0$ if $\cos \theta<0$. However, this is not the most convenient choice for the rest of our paper. We will assume that
\begin{equation} \label{e:nice-other-side}  B(r,\theta) = r^\gamma b(\cos \theta) \qquad \text{where } b(\cos \theta) \approx |\sin \theta|^{\gamma+\nu+1} \qquad \text{if } \cos \theta < 0.
\end{equation}
Note that for any $B$ which satisfies \eqref{e:B-assumption}, there is always a $\tilde B$ which also satisfies \eqref{e:nice-other-side}. Thus, we can assume \eqref{e:nice-other-side} without loss of generality. The purpose of \eqref{e:nice-other-side} is a purely technical convenience which will be apparent in the proof of Corollary \ref{c:Kf-approx}. With this choice, we have
\[ B(|v-v_\ast|,\theta) \approx \frac{|v-v_\ast|^{N-2} |v-v_\ast'|^{\gamma+\nu+1}}{|v-v'|^{N-1+\nu}}.\]

%\subsection{Basic surgery of the quadratic form}

Starting from \eqref{e:Boltzmann-Q}, we add and subtract $f(v_\ast') g(v)$ to the factor in the left and split $Q(f,f)$ between two terms which we call $Q_1$ and $Q_2$.

\begin{align}
Q_1(f,g) &= \int_{\R^N} \int_{\partial B_1} \left(g(v') - g(v) \right) f(v'_\ast) B\left( r,\theta \right) \dd \sigma \dd v_\ast. \label{e:Q1} \\
Q_2(f,g) &= g(v) \int_{\R^N} \int_{\partial B_1} \left( f(v'_\ast) - f(v_\ast) \right) B\left( r,\theta \right) \dd \sigma \dd v_\ast. \label{e:Q2}
\end{align}

These two terms for $Q$ have been considered in \cite{villani2002review}, chapter 2, section 6.2. As it is pointed out there, when $g$ is smooth ($g \in C^2$ would certainly be enough) and $f$ is bounded, $Q_1$ and $Q_2$ are well defined even in the non-cutoff case.

The way we understand these two terms is the following. For any given $f$, the first term $Q_1$ is an integro-differential operator applied to $g$. In the non-cutoff case its kernel is in the class described in section \ref{s:pie} and its parameters depend only on the macroscopic values of $f$. Thus, we can eventually apply the general results for parabolic integral equations of Theorems \ref{t:weak-harnack} and \ref{t:Calpha-general} and derive a regularization for the solution based on this term.

The second term $Q_2$ should be considered a lower order term. In many cases, we will be able to show that this term is bounded.

In the following sections, we provide a detailed analysis of both terms. The formulas we derive are closely related to similar expressions obtained in \cite{alexandre2000around}. See also section 3 in \cite{MR2556715}.
\section{The term $Q_1$}

\begin{lemma} \label{l:expression-for-Q1}
The term $Q_1$ can be rewritten in the following form
\begin{equation} \label{e:R1}
Q_1(f,g) = \int_{\R^N} (g(v')-g(v)) K_f(t,v,v') \dd v',
\end{equation}
where
\[ K_f(t,v,v') = \frac{2^{N-1}}{|v'-v|} \int_{\{w: w \cdot (v'-v) = 0\}} f(t,v+w) B \left( r, \theta \right) r^{-N+2} \dd w. \]
Note that the integral is supported for $w$ on the hyperplane perpendicular to $v'-v$. We now have.
\begin{align*}
r &= \sqrt{ |v'-v|^2 + |w|^2 }, \\
\cos (\theta/2) &= |w| / r, \\
v_\ast' &= v + w, \\
v_\ast &= v' + w.
\end{align*}
\end{lemma}

The proof consists of an elementary (but somewhat complicated) change of variables. It is explained in Lemma \ref{l:change-of-variables}, which we write in the appendix.

Lemma \ref{l:expression-for-Q1} describes how to obtain the kernel $K_f(v,v')$ in terms of a function $f: \R^N \to \R$. The $t$ dependence in $K_f(t,v,v')$ comes from the fact that $K_f(t,\cdot,\cdot)$ is computed using $f(t,\cdot)$. In the following lemmas, we omit the $t$ variable because the computations concern a fixed time $t$.

\begin{cor} \label{c:Kf-approx}
Assuming \eqref{e:nice-other-side}, then
\begin{equation}\label{e:Kg-approx}
K_f(v,v') \approx \left( \int_{\left\{w \cdot (v'-v) = 0 \right \}} f(v+w)\ |w|^{\gamma + \nu + 1}  \dd w \right) |v'-v|^{-N-\nu}. 
\end{equation}
\end{cor}

\begin{proof}
Using Lemma \ref{l:expression-for-Q1}
\[ K_f(v,v') = \frac{2^{N-1}}{|v'-v|} \int_{\{w: w \cdot (v'-v) = 0\}} f(v+w) r^{-N+2+\gamma} b(\cos \theta) \dd w.\]

Recall that we have $|w| = r \cos(\theta/2)$ and $|v'-v| = r \sin(\theta/2)$.

We analyse two cases: when $\cos(\theta) > 0$ and when $\cos(\theta)<0$. 

If $\cos(\theta)>0$, then $|w| \approx r$ and $b(\cos \theta) \approx \sin(\theta)^{-N+1-\nu} \approx \sin(\theta/2)^{-N+1-\nu}$. If $\cos(\theta)<0$, then $|v'-v| \approx r$, $|w| = r \cos(\theta/2)$, and using \eqref{e:nice-other-side}, $b(\cos \theta) \approx \sin(\theta)^{\gamma+\nu+1} \approx \cos(\theta/2)^{\gamma+\nu+1}$. Either way, we have
\[ r^{-N+2+\gamma} b(\cos \theta) \approx |w|^{\gamma+\nu+1} |v'-v|^{-N+1-\nu}.\]
And therefore,
\begin{align*}
K_f(v,v')
& \approx \left( \int_{\left\{w \cdot (v'-v) = 0 \right \}} f(v+w)\ |w|^{\gamma + \nu + 1}  \dd w \right) |v'-v|^{-N-\nu}.
\end{align*}
\end{proof}

Note that $K_f$ is a symmetric kernel in the sense that $K_f(v,v+y)=K_f(v,v-y)$. %We will prove that the operator $g \mapsto Q_1(f,g)$ belongs to the class of integro-differential operators described in \cite{schwab}, whose parameters $\mu$, $\lambda$ and $\Lambda$ depend only on the macroscopic values of $f$ (mass, energy and entropy).
%The main 
%\[K(f,v,v') \approx A \left(f,v, \frac{v'-v}{|v'-v|} \right) |v'-v|^{-N-\nu}.\] 
The values of $K_f$, and in particular whether it is bounded above or below for the different values of $v$ and $v'$ will depend on $f$. We see that $K_f(t,v,v')$ will be positive when $f(t,v+w)>0$ for some values of $w$ on the hyperplane $\{w \cdot (v'-v) = 0\}$. 

It is now not surprising that typically $K_f \approx |v'-v|^{-N-\nu}$.
We would be able prove this precise bound for $K_f$ if we knew that, for example, $f$ is bounded in between two Maxwellians. With this pointwise estimates for $K_f$, we would be able to apply the theory from \cite{luislecturenotes}, section 3.2. However, we do not want to make such strong assumptions on $f$. The only bounds we are willing to assume on $f$ concern the values of its mass, energy and entropy.

The following lemma will give us the upper bound \eqref{e:K-bound-above} for $K_f$.

\begin{lemma} \label{l:K-above}
Let $f : \R^n \to \R$ and $R>0$. Assume that
\begin{align*}
\sup_{v \in B_R} \int_{\R^N} |w|^{\gamma+\nu} f(v+w) \dd v &\leq K_0
\end{align*}
Then, there exists a constant $\Lambda$ such that
\[ \int_{B_{2r} \setminus B_r} K_f(v,v+z) \dd z \leq \Lambda r^{-\nu}.\]
for all $r>0$ and $v \in B_R$.

The value of $\Lambda$ depends only on $K_0$ and the cross section $B$.
\end{lemma}

\begin{proof}
Using Corollary \ref{c:Kf-approx}, polar coordinates, and Proposition \ref{p:dual-polar}, we estimate
\begin{align*}
\int_{B_{2r} \setminus B_r} K_f(v,v+z) \dd z 
&\leq C \left( \int_r^{2r} \int_{z \in \partial B_1} \left( \int_{\left\{ w: w \cdot z = 0 \right \}} f(v+w)\ |w|^{\gamma + \nu + 1}  \dd w \right) \rho^{N-1} \dd S(z) \dd \rho \right) r^{-N-\nu}, \\
&= C \int_{z \in \partial B_1} \left( \int_{\left\{ w: w \cdot z = 0 \right \}} f(v+w)\ |w|^{\gamma + \nu + 1}  \dd w \right) \dd S(z) \ r^{-\nu}, \\
&= C \left(\int_{\R^N} f(v+w) |w|^{\gamma+\nu} \dd w \right) r^{-\nu}\\
&\leq \Lambda r^{-\nu}
\end{align*}
%For the last inequality, we used that for all $v \in \R^N$
%\[ \left(\int_{\R^N} f(v+w) |w|^{\gamma+\nu} \dd w \right) \leq C(M_0,K_0).\]
\end{proof}

\begin{cor}
Let $0 \leq \gamma+\nu$, $R>0$ and $f: \R^N \to \R$ non negative such that
\begin{align*}
\int_{\R^N} f(v) \dd v &\leq M_0, \\
\int_{\R^N} |v|^{\gamma+\nu} f(v) \dd v &\leq K_0
\end{align*}
Then, there is a constant $\Lambda$ such that
\[ \int_{B_{2r} \setminus B_r} K_f(v,v+z) \dd z \leq \Lambda r^{-\nu}.\]
for all $r>0$ and $v \in B_R$.
\end{cor}

\begin{proof}
Since $f$ has a finite $\gamma+\nu$-moment, then, for all $v \in B_R$
\[ \int_{\R^N} |w|^{\gamma+\nu} f(v+w) \dd w \leq K(M_0,K_0,R),\]
and Lemma \ref{l:K-above} applies.
\end{proof}

\begin{remark}
If $\gamma+\nu < 2$, then we can bound the $\gamma+\nu$-moment in the previous corollary in terms of $M_0$ and an upper bound for the second moment $E_0$.
\end{remark}

In order to obtain the bound below for $K_f$ given in \eqref{e:K-bound-below}, we need to prove the following lemma first.

\begin{lemma} \label{l:lifted-set}
Let $f :\R^N \to R$ be non negative and
\begin{align*}
M_1 \leq \int_{\R^N} f(v) \dd v &\leq M_0, \\
\int_{\R^N} |v|^2 f(v) \dd v &\leq E_0, \\
\int_{\R^N} f(v) \ \log f(v) \dd v &\leq H_0.
\end{align*}
There exists an $r>0$, $\ell>0$ and $m>0$ depending on $M_0$, $M_1$, $E_0$ and $H_0$ such that
\[ |\{v : f(v)>\ell\} \cap B_r| \geq m \]
\end{lemma}

\begin{proof}
It is classical that, depending of $M_0$, $E_0$ and $H_0$, there is a constant $\tilde H_0$ such that
\[ \int_{\R^N} f(v) |\log f(v)| \dd v \leq \tilde H_0.\]

Since
\[ E_0 \geq \int_{\R^N} |v|^2 f(v) \dd v,\]
then, for all $r>0$, we clearly have
\[ E_0 \geq r^2 \int_{\R^N \setminus B_r} f(v) \dd v.\]
Let us choose $r>0$ large enough such that $E_0/r^2 < M_1/2$. Therefore
\[ \int_{\R^N \setminus B_r} f(v) \dd v \leq \frac 12 M_1 \leq \frac 12 \int_{\R^N} f(v) \dd v.\]
Therefore, 
\[ \int_{B_r} f(v) \dd v \geq \frac 12 M_1.\]
Let $\ell$ be such that $\ell |B_r| < \frac 14 M_1$. Thus
\[ \int_{B_r \cap \{f > \ell\}} f(v) \dd v \geq \frac 14 M_1.\]
We want to argue now that $|B_r \cap \{f > \ell\}| \geq m$ for some constant $m>0$. Indeed, let $T>0$ be such that $T |B_r \cap \{f > \ell\}| = \frac 18 M_1$. Then, on one hand, we have
\[ \int_{B_r \cap \{f > T\}} f(v) \dd v \geq
\frac 14 M_1 - \int_{B_r \cap \{T \geq f > \ell \}} f(v) \dd v \geq
 \frac 18 M_1.\]
On the other hand,
\[ \tilde H_0 \geq \int_{B_r \cap \{f > T\}} f(v) \log f(v) \dd v \geq \frac 18 M_1 \log (T).\]
Therefore, $T \leq \exp(8 H_0 / M_1)$. In particular,
\[|B_r \cap \{f > \ell\}| = \frac {M_1} {8T} \geq \frac{M_1}{8} \exp(-8 \tilde H_0 / M_1).\]
\end{proof}

\begin{remark}
In the previous lemma $r$ and $\ell$ depend on $M_1$ and $E_0$, whereas $m$ depends on $M_1$ and $H_0$.
\end{remark}

The following lemma gives us \eqref{e:K-bound-below} for $K_f$, where $K_f$ is the one from Lemma \ref{l:expression-for-Q1}.

\begin{lemma} \label{l:K-below}
Let $f: \R^N \to R$ be a non negative function such that
\begin{align*}
M_1 \leq \int_{\R^N} f(v) \dd v &\leq M_0, \\
\int_{\R^N} |v|^2 f(v) \dd v &\leq E_0, \\
\int_{\R^N} f(v) \ \log f(v) \dd v &\leq H_0.
\end{align*}
Then, for any $R>0$, there exists constants $\lambda$ and $\mu$ (depending on $M_0$, $M_1$, $E_0$, $H_0$ and $R$) such that for all $v \in B_R$, there exists a symmetric subset of the unit sphere $A = A(v) \subset \partial B_1$ such that
\begin{itemize}
\item $|A| > \mu$, where $|A|$ stands for the $(N-1)$-Haussdorff measure.
\item $K_f(v,v') \geq \lambda |v'-v|^{-N-\nu}$, every time $(v'-v)/|v'-v| \in A$.
\end{itemize}
\end{lemma}

\begin{proof}
From Lemma \ref{l:lifted-set}, there exists a constant $\ell > 0$, $m>0$ and $r>0$ such that 
\[ |\{v : f(v)>\ell\} \cap B_r| \geq m. \]
Let $S = \{v : f(v)>\ell\} \cap B_r$. We trivially have $f(v) \geq \ell \one_S(v)$. In particular
\[ K_f(v,v') \geq \ell K_{\one_S}(v,v').\]
In order to do this proof, we will prove the lower bound for the kernel $K_{\one_S}(v,v')$ instead.

Recall
\[ K_{\one_S}(v,v') \geq c \ell \left( \int_{\left\{w \cdot (v'-v) = 0 \right \}} \one_S(v+w)\ |w|^{\gamma + \nu + 1}  \dd w \right) |v'-v|^{-N-\nu}. \]

Note that $S \subset B_r$. The integrand in the right hand side can be non zero only when $v+w \in B_r$. In particular, for every fixed $v$, $|w|$ remains bounded and the whole integral will be bounded by a constant depending only on $M_0$, $M_1$, $E_0$ and $H_0$.
\begin{equation} \label{e:M}  \left( \int_{\left\{ w: w \cdot z = 0 \right \}} \one_S(v+w)\ |w|^{\gamma + \nu + 1}  \dd w \right) \leq (r+R)^{\gamma+\nu+1} |B_r| \leq C.
\end{equation}

Note that the integral in the right hand side depends on $v$ and $(v'-v)/|v'-v|$ only. Using Proposition \ref{p:dual-polar},
\[ \int_{\sigma \in \partial B_1} \left( \int_{\left\{ w: w \cdot \sigma = 0 \right \}} \one_S(v+w)\ |w|^{\gamma + \nu + 1}  \dd w \right) \dd \sigma = \int_{\R^N} \one_S(z) |z-v|^{\gamma+\nu} \dd z = \int_{S} |z-v|^{\gamma+\nu} \dd z.\]
Since $|S| \geq m$, then there is some small constant $\kappa$ so that
\begin{equation} \label{e:kappa}  \int_{\sigma \in \partial B_1} \left( \int_{\left\{ w: w \cdot \sigma = 0 \right \}} \one_S(v+w)\ |w|^{\gamma + \nu + 1}  \dd w \right) \dd \sigma \geq \kappa > 0.
\end{equation}
Combining \eqref{e:M} with \eqref{e:kappa}, we find that there is a set $A \subset \partial B_1$ of values of $\sigma$ so that the integrand is bounded below by a constant $\lambda$. That is
\[ \left( \int_{\left\{ w: w \cdot \sigma = 0 \right \}} \one_S(v+w)\ |w|^{\gamma + \nu + 1}  \dd w \right) \geq \lambda \qquad \text{if } \sigma \in A.\]
This concludes the proof.
\end{proof}

\section{The term $Q_2$}

%Recall that $B(r,\theta) = r^\gamma b(\theta)$, where $\gamma \in (-N,2)$ and
%\begin{align*}
%b(\theta) &\approx |\sin \theta|^{-N+1-\nu} && \text{where } \cos \theta > 0, \\
%b(\theta) &\approx |\sin \theta|^{1+\gamma+\nu} && \text{where } \cos \theta < 0. \\
%\end{align*}

In this section we analyse the term $Q_2$. This term is understood as a lower order term compared to $Q_1$. This point of view is not new since it follows the same philosophy as the cancellation lemma in \cite{alexandre2000entropy}. Indeed, Lemma \ref{l:expression-for-Q2} coincides with Lemma 1 in \cite{alexandre2000entropy}. We include the full proof here for completeness (and because it is short and easy anyway).

The term $Q_2$ is reduced to a simpler formula in the following lemma. The proof is an elementary change of variables. In 2D it is much easier to understand than in higher dimension. The proof we write below is general.

\begin{lemma} \label{l:expression-for-Q2} For any cross section $B$, we can rewrite the term $Q_2$ from \eqref{e:Q2} as
\begin{equation} \label{e:Q2prime} Q_2(f,g) = \left(\int_{\R^N} f(v-w) \tilde B(|w|) \dd w \right) g(v). 
\end{equation}
Where, %\comment{There is something wrong in this expression}
%\[ \tilde B(r) := \int_{\partial B_1} \left(2^N (1-\sigma \cdot e)^{-1} B\left( \frac {2r} {1-\sigma \cdot e} , \theta \right) - B(r, \theta) \right) \dd \sigma.\]
\[ \tilde B(r) := \int_{\partial B_1} \left( 2^{N/2} (1-\sigma \cdot e)^{-N/2} B\left( \frac {\sqrt 2 r} {\sqrt{1-\sigma \cdot e}} , \theta \right) - B(r, \theta) \right) \dd \sigma > 0.\]
Here $e$ is any unit vector (the value of $\tilde B$ is independent of this choice) and $\cos \theta = \sigma \cdot e$.

\end{lemma}

\begin{proof}
By a standard approximation argument, we can assume that $B$ is integrable in both $r$ and $\theta$.

We divide the expression for $Q_2$ into two terms
\begin{align*} Q_2 &= \left( \int_{\R^N} \int_{\partial B_1} f(v_\ast') B(r,\theta) \dd \sigma \dd v_\ast - \int_{\R^N} f(v_\ast) \left( \int_{\partial B_1} B(r,\theta) \dd \sigma \right) \dd v_\ast \right) g(v).
\end{align*}
We leave the second term as it is and use Lemma \ref{l:change-of-variables} for the first term.

The first term inside the parenthesis becomes
\[ \int_{\R^N} \int_{\{w\cdot v'_\ast = 0\}} f(v'_\ast) B(r,\theta) \frac{2^{N-1}}{|v'_\ast| r^{N-2}} \dd w \dd v'_\ast.\]
Here $r^2 = |v'_\ast|^2 + |w|^2$ and $\sin (\theta/2) = |w|/|v_\ast'|$.

We define $\psi \in \partial B_1$ to be the normalized stereographic projection of $v'_\ast + w$ onto the sphere with diameter $v'_\ast$. Note that this $\psi$ is different from the original $\sigma$.
\[ \psi := 2|v'_\ast| \frac{v_\ast' + w}{|v_\ast' + w|^2} - \frac{v'_\ast}{|v'_\ast|}.\]
(note this is always a unit vector)
We verify the Jacobian of the (conformal) change of variables by a straight forward computation
\[ \dd \psi = \left( \frac{2^{N-1} |v_\ast'|^{N-1}}{r^{2N-2}} \right) \dd w.\]
Therefore, the first term becomes
\[ \int_{\R^N} f(v'_\ast) \left( \int_{\partial B_1}  B(r,\theta) \frac{r^N}{|v'_\ast|^N} \dd \psi \right) \dd v'_\ast.\]
Let $e = v'_\ast / |v'_\ast|$. We have $\cos \theta = e \cdot \psi$ and $r = |v'_\ast| / \cos(\theta/2) = |v'_\ast| / \sqrt{(1+\psi \cdot e)/2}$. We get the following expression for the inner integral
\[  \int_{\partial B_1}  B(r,\theta) \frac{r^N}{|v'_\ast|^N} \dd \psi = \int_{\partial B_1} B\left( \frac{\sqrt 2 |v'_\ast|} {\sqrt{1+\psi \cdot e}}, \cos^{-1} \left(\psi \cdot e\right) \right) \frac{2^{N/2}} {(1+\psi \cdot e)^{N/2}} \dd \psi.\]
Notice that this expression is independent of $e \in \partial B_1$. We finally rename the variables $v_\ast'$ and $\psi$ as $w$ and $\sigma$ and finish the proof with a straight forward replacement.
\end{proof}

\begin{lemma} \label{l:Btilde} Assuming that $B$ has the form \eqref{e:nice-other-side}, then $\tilde B(r) = Cr^\gamma$ for some positive constant $C$ depending on $\gamma$ and $b$.
\end{lemma}

\begin{proof}
In the case that $B(r,\theta)$ has the form \eqref{e:nice-other-side}. Applying Lemma \ref{l:expression-for-Q2} We get
\begin{align}
\tilde B(r) &= \int_{\partial B_1} \left( 2^{N/2} (1+\sigma \cdot e)^{-N/2} \left( \frac {\sqrt 2 r} {\sqrt{1+\sigma \cdot e}} \right)^\gamma  - r^\gamma \right) b(\cos \theta) \dd \sigma, \\
&= r^\gamma \int_{\partial B_1} \left( \frac{ 2^{{(N+\gamma)}/2} }{ (1+\sigma \cdot e)^{(N+\gamma)/2}} - 1 \right) b(\cos \theta) \dd \sigma. \label{e:tildeB}
\end{align}
As before, we assume that $b$ is a nice bounded function and extend the result to non integrable cross sections $B$ by approximation. This is a standard procedure, so we will just concentrate on the bounds.

The first factor is bounded except around $\sigma \approx -e$, whereas $b(\cos(\theta))$ is bounded except when $\sigma \approx e$ (i.e. $\theta \approx 0$). 

Let us analyse the integrand around both points. Starting with the latter,
\[ b(\cos \theta) \leq C |\theta|^{-N+1-\nu} \ \text{ and } \ \left( \frac{ 2^{{(N+\gamma)}/2} }{ (1+\sigma \cdot e)^{(N+\gamma)/2}} - 1 \right) \leq C |\theta|^2, \text{ for } \sigma \approx e.\]
Therefore, we bound the integrand by the integrable quantity
\[ \left( \frac{ 2^{{(N+\gamma)}/2} }{ (1+\sigma \cdot e)^{(N+\gamma)/2}} - 1 \right) b(\cos \theta) \leq C |\sin \theta|^{-N+3-\nu}.\]
The right hand side is integrable on the sphere $\partial B_1$ around $\sigma = e$ since $\nu < 2$.

We now analyse the integrand around the point $\sigma \cdot e \approx -1$. Using the assumption \eqref{e:nice-other-side},
\[ b(\cos \theta) \leq C |\sin \theta|^{\gamma+\nu+1} \ \text{ and } \ \left( \frac{ 2^{{(N+\gamma)}/2} }{ (1+\sigma \cdot e)^{(N+\gamma)/2}} - 1 \right) \leq C |\sin \theta|^{-N-\gamma}, \text{ for } \sigma \approx -e.\]
Therefore, we bound the integrand by the quantity
\[ \left( \frac{ 2^{{(N+\gamma)}/2} }{ (1+\sigma \cdot e)^{(N+\gamma)/2}} - 1 \right) b(\cos \theta) \leq C |\sin \theta|^{-N+1+\nu}.\]
This bound is integrable on $\partial B_1$ around $\sigma = -e$ because $\nu > 0$.

Therefore, the integral in \eqref{e:tildeB} gives us a well defined positive number. We obtain that $\tilde B(r) = C r^\gamma$ for some constant $C$ depending on $b$ and $\gamma$.
\end{proof}

In the case $B(r,\theta) = r^\gamma b(\theta)$, then $\tilde B(r) = \lambda r^\gamma$. The factor $\lambda$ depends on $\gamma$ and converges to zero as $\gamma \to -N$. Note that this does not mean that $Q_2 \to 0$ as $\gamma \to -N$ because $r^{-N}$ is not integrable in $\R^N$. %{\color{red} In fact, for the limit case $\gamma = -N$, it can be seen that the term $Q_2$ equals $C f(v)^2$ for some constant $C$.} 

Since $\tilde B$ is locally integrable and $Q_2$ has the form \eqref{e:Q2prime}, we abuse notation by writing 
\[Q_2(f,g) (v) = \left(f \ast \tilde B \right) (v) g(v),\]
where we understand $\tilde B(x) = \tilde B(|x|)$ to be a radially symmetric function in $\R^N$.

%\begin{remark} {\color{red} IS THIS TRUE?}
%The following alternative form for $\tilde B$ might be useful at some point
%\[ \tilde B(r) = 2^{N-1} \int_{\R^{N-1}} \left(1+|z|^2\right)^{-N/2+1} \left( B \left( \left( \sqrt{1+|z|^2} \right) r, \theta \right) - B(r,\theta) \right)\dd z, ????? \]
%where $\tan \theta/2 = |z|$.
%\end{remark}

\begin{remark} When $\gamma < 0$, we cannot get an immediate upper bound of $\tilde B \ast f$ in terms of the macroscopic values of $f$. Indeed, we need a higher integrability condition so that $\tilde B \ast f$ is bounded. For example, if $f \in L^\infty \cap L^1$, we could conclude.
\[ |\tilde B \ast f| \leq C |f|_{L^\infty}^{-\gamma/N} |f|_{L^1}^{1+\gamma/N}.\]
\end{remark}

\section{Lower bound}

\begin{thm} \label{t:lower-bound} Let $f$ be a solution of \eqref{e:space-homogeneous-boltzmann} and $R>0$. Assume
\begin{align*}
\int_{\R^N} f(0,v) \log f(0,v) \dd v &\leq H_0, \\
\int_{\R^N} |v|^2 f(0,v) \dd v &\leq E_0, \\
0< M_1 \leq \int_{\R^N} f(0,v) \dd v &\leq M_0, \\
\sup_{(t,v) \in [0,T]\times B_{2R}} \int_{\R^N} |w|^{\gamma+\nu} f(t,v+w) \dd w &\leq K_0
\end{align*}
then there exists a constant $c = c(R,T,H_0, M_0, M_1, E_0, K_0)$ such that
\[ f(T,v) \geq c,\]
for all $v \in B_R$.
\end{thm}

Note that when $0 \leq \gamma+\nu \leq 2$, the value of $K_0$ is controlled by $M_0$ and $E_0$, and therefore the last inequality in the assumptions in redundant. If $\gamma+\nu > 2$, in particular $\gamma>0$, and thus there is instantaneous creation of infinite moments (see \cite{villani2002review}, chapter 2, section 2). Therefore, the last inequality, involving $K_0$, is redundant whenever $\gamma+\nu \geq 0$.

\begin{proof}
We re-write the equation \eqref{e:space-homogeneous-boltzmann} as
\[ f_t = Q_1(f,f) + Q_2(f,f).\]
From Lemma \ref{l:expression-for-Q2}, $Q_2(f,f) = f \cdot (\tilde B \ast f)$, with $\tilde B >0$. Thus, $Q_2 \geq 0$. We have
\[ f_t \geq Q_1(f,f).\]
Thus,
\[ f_t \geq Q_1(f,f) = \int_{\R^N} (f(t,v')-f(t,v)) K_f(t,v,v') \dd v'.\]
The kernel $K_f$ satisfies the hypothesis of Theorem \ref{t:weak-harnack} because of Lemma \ref{l:K-above} and Lemma \ref{l:K-below}.

Using Theorem \ref{t:weak-harnack} together with Lemma \ref{l:lifted-set}, we finish the proof.
\end{proof}

\begin{remark}
The value of $c$ necessarily decays as $R \to \infty$ because the result must be compatible with the Maxwellian solution. Therefore, the best we could expect to prove is that $c(R) \geq c_1 \exp(-C_2 R^2)$ for $c_1$ and $C_2$ depending on the other parameters. The explicit form of the constant $c$ in the theorem above is, in theory, computable following the proofs in \cite{schwab}. In practice, it is difficult to write its explicit form since these proofs are quite involved. There is no reason to believe that the method of this proof will provide an optimal lower bound.
\end{remark}

\section{Bounds in $L^\infty$}

In this section we obtain an a priori estimate in $L^\infty$ for the solution $f$ of the Boltzmann equation. We use maximum principle techniques taking advantage of the decomposition of the Boltzmann operator given in section \ref{s:twoterms}. This section does not apply results from \cite{schwab}. The proofs are self contained. The main technique has some similarity with \cite{gamba2009upper}. We state the result in general for the space in-homogeneous Boltzmann equation.

We need the following lemma, which is just a refined version of Lemma \ref{l:K-below}.

\begin{lemma} \label{l:K-below-v-large}
Let $f : \R^N \to \R$ be a non negative function such that
\begin{align*}
0 < M_1 \leq \int_{\R^N} f(v) \dd v &\leq M_0, \\
\int_{\R^N} |v|^2 f(v) \dd v &\leq E_0, \\
\int_{\R^N} f(v) \ \log f(v) \dd v &\leq H_0.
\end{align*}
Then, for any $v \in \R^N$, there exists a symmetric subset of the unit sphere $A = A(v) \subset \partial B_1$ such that
\begin{itemize}
\item $|A| \geq \mu / (1+|v|)$, where $|A|$ stands for the $N-1$-Haussdorff measure of $A$.
\item $K_f(v,v') \geq \lambda (1+|v|)^{1+\gamma+\nu} |v-v'|^{-N-\nu}$ every time $(v'-v)/|v'-v| \in A$.
\item For every $\sigma \in A$, $|\sigma \cdot v| \leq C$.
\end{itemize}
The constants $\mu$, $\lambda$ and $C$ depend on the cross section $B$ and the values of $M_0$, $M_1$, $E_0$ and $H_0$ only.
\end{lemma}

Note that the condition $|\sigma \cdot v| \leq C$ means that $\sigma$ belongs to a band of width at most $C/|v|$ around the big circle on $\partial B_1$ which is perpendicular to $v$.

\begin{proof}
This is essentially the same result as Lemma \ref{l:K-below}. Indeed, the statement of Lemma \ref{l:K-below} already covers the result when $|v|$ is bounded by an arbitrary constant $R>0$. We are left with the analysis of the result for large values of $|v|$.

From Lemma \ref{l:lifted-set}, there is a constant $\ell > 0$, $m>0$ and $r>0$ such that
\[ |\{v : f(v) > \ell \} \cap B_r| \geq m.\]
Like in the proof of Lemma \ref{l:K-below}, we call $S = \{v : f(v) > \ell \} \cap B_r$ and observe that
\begin{equation} \label{e:klbv1}  K_f(v,v') \geq \ell K_{\one_S}(v,v') \geq c \ell \left( \int_{\{w \cdot (v'-v) = 0\}} \one_S(v+w) |w|^{\gamma+\nu+1} \dd w \right) |v'-v|^{-N-\nu}.
\end{equation}

For $|v|$ sufficiently large, we have that $|v|>r$ and $S \subset B_r$. In particular $v \notin S$. We need to identify the set of directions $\sigma$ for which, if $v'-v$ is parallel to $\sigma$ then $w \cdot (v'-v) = 0$ for a set of values of $w$ of positive measure such that $v+w \in S$. We know that the set $|S|>\mu$. We define $A(v) = \{ \sigma \in \partial B_1 :   |\{ w : w \cdot \sigma = 0\} \cap S| > \delta \}$.
 Therefore, for some $\delta>0$ and $c>0$, the measure of this $A$ is bounded by
\[ |A| = |\{ \sigma \in \partial B_1 :   |\{ w : w \cdot \sigma = 0\} \cap S| > \delta \} | > \frac{c \mu} {|v|}.\]
Indeed, it is easy to observe that for every big circle on $\partial B_1$ which goes through the point $v/|v|$, the set $A$ covers a positive proportion of the band of width $2r/|v|$ centered at the point perpendicular to $v/|v|$. In particular the last item of the lemma holds.
\medskip
\medskip

\begin{center}
\setlength{\unitlength}{1in} 
\begin{picture}(2.11111,2)
\put(0,0){\includegraphics[height=2in]{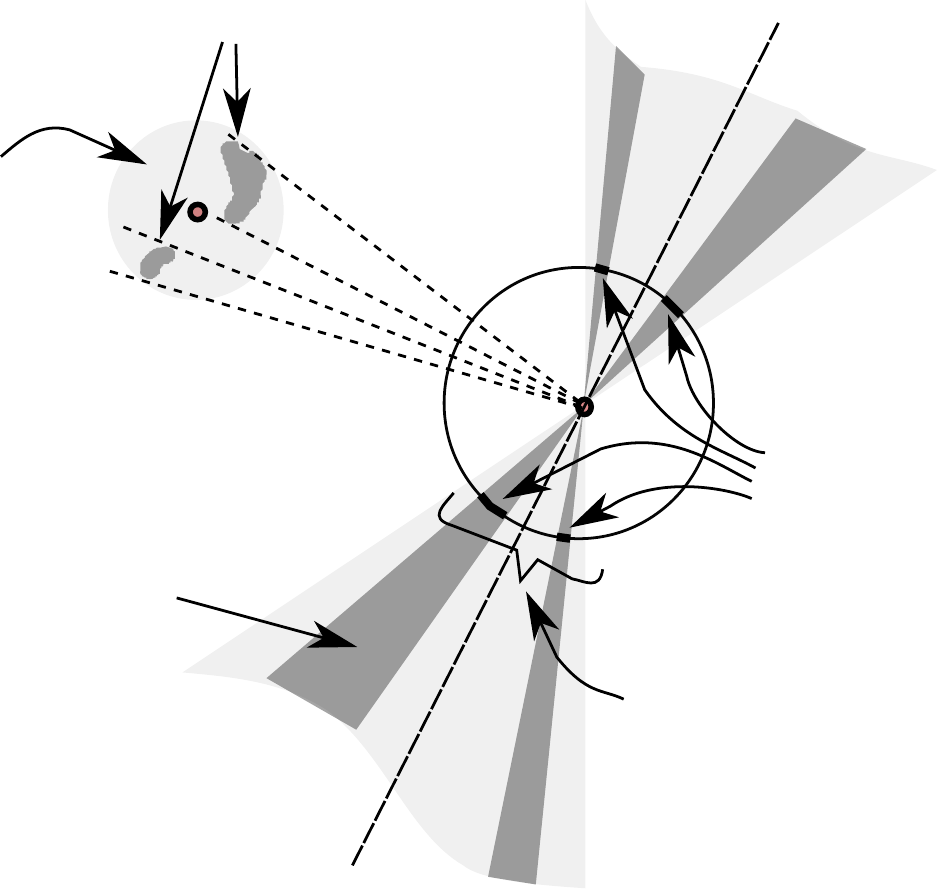}}
\put(0.45,1.94){$S$}
\put(-0.174,1.54){$B_r$}
\put(1.74,0.882){The set $A$ on $\partial B_1$}
\put(1.46,0.41){$A$ is contained inside a band}
\put(1.46,0.216){of width at most $C/|v|$}
\put(-1.58,0.66){Here $K_f(v,v')$ is bounded below}
\end{picture}
\end{center}

When $|v|$ is large and $(v'-v)/|v'-v| \in A$, the lower bound on $K_f$ easily follows from \eqref{e:klbv1} observing that $|w| \approx |v|$ every time $v+w \in S$.
\end{proof}

%The following lemma is a more or less obvious observation. For clarity, it is convenient to make it explicit before it is used in the main proof of the $L^\infty$ bound.

%\begin{lemma}
%Assume $f$ achieves its maximum $m$ at a point $v \in \R^N$ and
%\[ \int_{\R^N} f(v) \dd v \leq M_0.\]
%Assume also that $K$ is a kernel and $\Xi \subset \R^N$ is a cone such that 
%\[ K(v,v') \geq \lambda |v-v'|^{-N-\nu} \text{ every time } (v'-v) \in \Xi.\]
%Let $\mu$ be the $N-1$ dimensional measure of $\Xi \cap \partial B_1$. Then,
%\[ \int_{\R^N} (f(v') - f(v)) K(v,v') \dd v' \leq -c m^{1+\nu/N}.\]
%\end{lemma}

It is relatively easy to see that $Q_1(f,f) \leq 0$ at the point $v$ where the maximum of $f$ is achieved. The following lemma helps us quantify how negative $Q_1(f,f)$ will be at its maximum point using the $L^p$ norm of $f$. Estimates similar to this are rather common for the fractional Laplacian or other simpler integral operators. See for example equations (3.4) and (3.5) in \cite{czubak}, or \cite{constantin2012nonlinear}.

\begin{lemma} \label{l:quantifyingintegralatmax}
We say that $\Xi$ is a cone centered at $v$ if  there is a subset $A$ of the unit sphere $\partial B_1$ so that $\Xi$ is given by
\[ \Xi = \set{ v' : \frac{v'-v}{|v'-v|} \in A}.\]
Assume that $\max_{\R^N} f = f(v) = m$ and $|A| \geq \mu > 0$. Then
\[ \int_{\Xi} (m-f(v')) |v-v'|^{-N-\nu} \dd v' \geq c \frac{m^{1+p\nu/N} \mu^{1+\nu/N} }{\left(\int_{\Xi} |f(v')|^p \dd v'\right)^{\nu/N}},\]
for some constant $c$ depending on $p$, $\nu$ and $N$.
\end{lemma}

\begin{proof}
Let $G = \{v' \in \Xi : f(v') \leq m/2 \}$. By Chebyshev's inequality,
\[ |\Xi \setminus G| = | \{v' \in \Xi : f(v') > m/2 \} | \leq 
\frac {2^p}{m^p} \int_{\Xi} |f(v')|^p \dd v'.\]
For any point $v' \in G$, we have $(m-f(v')) \geq m/2$. Therefore,
\[ \int_{\Xi} (m-f(v')) |v-v'|^{-N-\nu} \dd v' \geq \frac m 2 \int_{G} |v-v'|^{-N-\nu} \dd v'.\]
Given our upper bound for $|\Xi \setminus G|$, the minimum possible value of $\int_G |v-v'|^{-N-\nu} \dd v'$ is then given by locating $\Xi \setminus G$ where $|v-v'|^{-N-\nu}$ is the largest possible. That is, the worst case scenario is when $G = \Xi \setminus B_r$, choosing $r$ so that
$|\Xi \cap B_r| = 2^p/m^p \, \int_{\Xi} |f(v')|^p \dd v'$. This corresponds to
\[ r = \left( \frac{N 2^p}{\mu m^p} \int_{\Xi} |f(v')|^p \dd v' \right)^{1/N}.\]
Thus, we get
\begin{align*}
 \int_{\Xi} (m-f(v')) |v-v'|^{-N-\nu} \dd v' &\geq \frac m 2 \int_{G} |v-v'|^{-N-\nu} \dd v',\\
&\geq \frac m 2 \int_{\Xi \setminus B_r} |v-v'|^{-N-\nu} \dd v',\\
&= \frac m 2  \mu \frac 1 {\nu} r^{-\nu} = c m^{1+p \nu /N} \mu^{1+\nu/N} \left( \int_\Xi |f(v')|^p \dd v' \right)^{-\nu/N}.
\end{align*}

\end{proof}

\begin{thm} \label{t:Linfty}
Let $f : [0,\infty) \times \R^N \times \R^N \to \R$ be a solution of the Boltzmann equation \eqref{e:Boltzmann-inhomogeneous}. Assume that $f$ is periodic in $x$, and for all $x \in \R^N$ and $t \geq 0$, we have
\begin{align*}
0 < M_1 \leq \int_{\R^N} f(t,x,v) \dd v &\leq M_0, \\
\int_{\R^N} |v|^2 f(t,x,v) \dd v &\leq E_0, \\
\int_{\R^N} f(t,x,v) \log f(t,x,v) \dd v &\leq H_0.
\end{align*}
Moreover, if $\gamma > 2$, we need to assume a higher moment condition, which says that for all $x \in \R^N$ and $t \geq 0$,
\begin{equation} \label{e:extra-moment}  \int_{\R^N} f(t,x,v) |v|^\gamma \dd v \leq \tilde K_0. \qquad \text{\bf (only necessary if $\gamma>2$)}
\end{equation}
If $\gamma+\nu \leq 0$, we need to assume higher integrability. In this case, we also assume that there is some $p > N / (N+\nu+\gamma)$ so that for all $x \in \R^N$ and $t \geq 0$, and for $q = \max(0,1-\frac N \nu (\nu+\gamma))$,
\begin{equation} \label{e:extra-integrability}
\int_{\R^N} (1+|v|)^q f(t,x,v)^p \dd v \leq K_0 \qquad \text{\bf (only necessary if $\gamma+\nu \leq 0$)}
\end{equation}

Then, an $L^\infty$ bound holds
\[ \|f(t,\cdot)\|_{L^\infty} \leq C(t),\]
Here $C(t) = a + b t^{-\beta}$ for some constants $a$, $b$ and $\beta$ depending only on $E_0$, $M_0$, $M_1$, $H_0$, $\tilde K_0$ (if $\gamma > 2$) and $K_0$ (if $\gamma+\nu \leq 0$).
\end{thm}

\begin{remark}
Note that the assumption \eqref{e:extra-moment} is redundant when $\gamma \in [0,2]$ since in that case the $\gamma$-moment is controlled by $M_0$ and $E_0$. When $\gamma<0$, the assumption \eqref{e:extra-moment} is not necessary for the estimate.

Also, if $\gamma+\nu > 0$, the assumption \eqref{e:extra-integrability} is redundant, since it always holds for $p=1$ and $q \in [0,1]$, with $C$ depending on $M_0$ and $E_0$.
\end{remark}

\begin{proof}
For every $t \geq 0$, let us define $m(t)$ to be the maximum of $f(t,x,v)$ for $x,v \in \R^N$. Since we assume that $f$ is a classical solution, there must be at least one point $x \in \R^N$ and $v \in \R^N$ so that $f(t,x,v) = m(t)$. The periodicity of $f$ with respect to $x$ is used only for the existence of this point.

We claim that there are constants $a$ and $b$, depending only of the parameters in the Lemma, such that
\begin{equation} \label{e:ODEm}  m(t) < a + b t^{- \frac N {p\nu}}.
\end{equation}
Here $p$ is the parameter from \eqref{e:extra-integrability} and we take $p=1$ if $\gamma+\nu > 0$. The conclusion of the Theorem will follow with $C(t) = a + b t^{- \frac N {p\nu}}$.

Indeed, \eqref{e:ODEm} must hold for $t$ sufficiently small simply because the right hand side diverges as $t \to 0$. If \eqref{e:ODEm} was not true, then there would be a first time $t$ when the inequality is invalidated. For this time $t$, there would exist a point $x$ and $v$ such that
\begin{align}
f(t,x,v) &= m(t) = a + b t^{-N/\nu}, \label{e:mp1} \\
f_t(t,x,v) &\geq \frac{\dd}{\dd t} \left( a + b  t^{-N/\nu} \right) = - \frac {Nb} \nu t^{-1-N/\nu} = - \frac N \nu b^{\nu/N} m(t) (m(t)-a)^{\nu/N}, \label{e:mp2} \\
\grad_x f &= 0.
\end{align}
We will obtain a contradiction by evaluating the equation \eqref{e:Boltzmann-inhomogeneous} at the point $(t,x,v)$. Indeed, since $\grad_x f(t,x,v)=0$, we get
\[ f_t(t,x,v) = Q(f,f)(t,x,v).\]
The rest of the proof consists in obtaining an upper bound for $Q(f,f)(t,x,v)$ which contradicts \eqref{e:mp2}.

All the computations below are for a fixed value of $t$ and $x$. We omit writing it in order to keep the computation uncluttered. That is, we write $f(v) := f(t,x,v)$ and $m := m(t)$.

We use the decomposition $Q = Q_1 + Q_2$.
\[
f_t(v) = Q_1(f,f)(v) + Q_2(f,f)(v).\]

The proof is based on various estimates for $Q_1$ and $Q_2$. There are a few cases, depending on the range of values of $\gamma$ and $\nu$. Note that some of the estimates below are redundant. For example, a bound depending on $M_0$ is a particular case of the bound depending on $L^p$ when $p=1$. Yet, we cover all cases in detail below.

\noindent \textbf{$\circ$ Upper bound for $Q_2$ using $M_0$ and $E_0$ only.}

From Lemma \ref{l:expression-for-Q2}, we have that \[ Q_2(f,f)(v) = (\tilde B \ast f)(v) \ f(v).\]
Recall from Lemma \ref{l:Btilde} that $\tilde B(v) = c |v|^\gamma$. If $\gamma \in [0,2]$, then we would have $(\tilde B \ast f)(v) \leq C |v|^\gamma$, where $C$ depends only on $M_0$ and $E_0$.

If $\gamma > 2$, then we need to assume that $f$ has finite $\gamma$ moments in order to claim that $\tilde B \ast f \leq C |v|^\gamma$. This is the only point where \eqref{e:extra-moment} is used. In this case the constant $C$ depends on $M_0$ and the constant from \eqref{e:extra-moment}.

If $\gamma < 0$, then
\begin{align*}
(\tilde B \ast f)(v) &= \int_{\R^N} c |v-w|^\gamma f(w) \dd w, \\
&= \int_{B_r(v)} c |v-w|^\gamma f(w) \dd w + \int_{\R^N \setminus B_r(v)} c |v-w|^\gamma f(w) \dd w, \text{ for any value of } r>0,\\
&\leq \min_r \left\{ C m r^{N+\gamma} + C r^\gamma M_0 \right\} = C M_0^{1+\gamma/N} m^{-\gamma/N}
\end{align*}

Summarizing,
\begin{equation} \label{e:ub-Q2}  Q_2(f,f)(v) = (\tilde B \ast f)(v)  f(v) \leq \begin{cases}
C m (1+|v|)^\gamma & \text{if } \gamma \geq 0,\\
C m^{1-\gamma/N} & \text{if } \gamma<0,
\end{cases}
\end{equation}
where the constant $C$ depends on $M_0$ and $E_0$.

\noindent \textbf{Upper bound for $Q_2$ using $\|f\|_{L^p} \leq C$ for some $p>1$.}

Unfortunately, when $\gamma+\nu < 0$, we have not been able to finish this proof using \eqref{e:ub-Q2}. Instead, we estimate $\tilde B \ast f$ using $\|f\|_{L^p}$ for some $1 < p < N / (N+\gamma)$. Note that in this case $\gamma < 0$. We get

\begin{align*}
(\tilde B \ast f)(v) &= \int_{\R^N} c |v-w|^\gamma f(w) \dd w, \\
&= \int_{B_r(v)} c |v-w|^\gamma f(w) \dd w + \int_{\R^N \setminus B_r(v)} c |v-w|^\gamma f(w) \dd w, \text{ for any value of } r>0,\\
&\leq \min_r \left( C m r^{N+\gamma} + C \|f\|_{L^p} \left( \int_{\R^N \setminus B_r} |w|^{\gamma p'} \dd w \right)^{1/p'} \right), \\
&\leq C \min_r \left(  m r^{N+\gamma} + r^{\gamma + N/p'} \|f\|_{L^p} \right) \qquad \text{using that $\gamma p' < -N$},\\
&= C m^{1 - p \frac{N+\gamma}N } \qquad \text{ where the constant $C$ depends on $K_0$}.
\end{align*}
Therefore
\begin{equation} \label{e:ub-Q2-withLp}  Q_2(f,f) \leq C m^{2 - p \frac{N+\gamma} N }.
\end{equation}

\noindent \textbf{$\circ$ Lower bound for $Q_1$ - general considerations.}

From Lemma \ref{l:expression-for-Q1}, we have that
\[ Q_1(f,f)(v) = \int_{\R^N} K_f(v,v') \left(f(v') - f(v)\right) \dd v'.\]
Note that since $f(v)=m=\max f$, then the integrand is non positive everywhere. The term $Q_1(f,f)(v)$ will be negative. A good estimate of how negative $Q_1(f,f)(v)$ is, is what makes this proof possible.

Using Lemma \ref{l:K-below-v-large}, we know that there is a set of directions $A \subset \partial B_1$ (depending on $t$, $x$ and $v$) with measure bounded below by $c (1+|v|)^{-1}$ so that $K_f$ satisfies an appropriate bound below when $(v'-v)$ is parallel to some direction in $A$:
\[ K_f(v,v') \geq \lambda (1+|v|)^{1+\gamma+\nu} |v-v'|^{-N-\nu} \ \text{whenever } \frac {v'-v}{|v'-v|} \in A.\]

While $|A|$ tends to zero as $|v| \to \infty$, at the same time, since $1+\gamma+\nu >0$, the value of the lower bound on $K_f$ grows as $|v| \to \infty$. A balance between these two effects will give us the right estimate for all values of $v \in \R^N$. We estimate $Q_1(f,f)$ in two different ways depending on whether $v$ is small or large.

We write
\begin{equation} \label{e:Linfty-nc}  \begin{aligned}
Q_1(f,f)(v) &= \int_{\R^N} (f(v')-f(v)) K_f(v,v') \dd v', \\
&\leq -\lambda (1+|v|)^{1+\gamma+\nu} \left(
\int_{\{ (v'-v)/|v'-v| \in A \}} (m - f(v')) |v-v'|^{-N-\nu} \dd v' \right)
\end{aligned} 
\end{equation}

Let us call $\Xi$ the cone of non degenerate values for $v'$, that is
\[ \Xi := \set{ v' : \frac{v'-v}{|v'-v|} \in A}.\]
Thus,
\[ Q_1(f,f)(v) \leq -\lambda (1+|v|)^{1+\gamma+\nu} \int_{\Xi} (m-f(v')) |v-v'|^{-N-\nu} \dd v' .\]
According to Lemma \ref{l:K-below-v-large}, the cone $\Xi$ has an aperture which is bounded below by $c (1+|v|)^{-1}$ and moreover $|(v'-v) \cdot v| \leq C |v'-v|$ for all $v' \in \Xi$. Depending on this constant $C$, let us pick $R>0$ so that $|v'|>|v|/2$ if $v' \in \Xi$ and $|v|>R$.

\noindent \textbf{$\circ$ Lower bound for $Q_1$ for small values of $|v|$ using $M_0$.}

Let us start with the case $|v| < R$. In this case, the aperture of the cone $\Xi$ is uniformly bounded below (in measure). In order to get a quantitative estimate of the right hand side above, we use that $m \geq f$ everywhere and $\int f(v') \dd v' \leq M_0$. Applying Lemma \ref{l:quantifyingintegralatmax}, we get
\[ \int_\Xi (m-f(v')) |v-v'| dv' \geq c m^{1+\nu/N},\]
and therefore
\begin{equation} \label{e:ubq11}  Q_1(f,f)(v) \leq -c m^{1+\nu/N} \ \text{ every time } |v|<R.
\end{equation}

\noindent \textbf{$\circ$ Lower bound for $Q_1$ for small values of $|v|$ using $\|f\|_{L^p} \leq C$.}

If we had the extra piece of information that $\|f\|_{L^p} \leq C$ for some $p>1$, then we get a better estimate at this step. We will use this improved estimate in the case $\gamma+\nu \leq 0$. Indeed, applying Lemma \ref{l:quantifyingintegralatmax}, we get
\[ \int_\Xi (m-f(v')) |v-v'| dv' \geq c m^{1+p\nu/N},\]
and therefore
\begin{equation} \label{e:ubq12}  Q_1(f,f)(v) \leq -c m^{1+p\nu/N} \ \text{ every time } |v|<R,
\end{equation}
where the constant $c$ depends on $\|f\|_{L^p}$.

%\noindent \textbf{$\circ$ Lower bound for $Q_1$ for large values of $|v|$ using $M_0$.}

%For $|v|>R$, the opening of $\Xi$ is bounded below by $c/|v|$. In order to estimate the right hand side of \eqref{e:Linfty-nc} for $|v|>R$, we will argue similarly to the case $|v|\leq R$. As before, given that $m-f(v') \geq 0$ everywhere and the integral of $f$ is bounded, the smallest possible value of the right hand side of \eqref{e:Linfty-nc} takes place when we accumulate all the possible mass of $f$ as close as possible to $v$ so that the factor $|v-v'|^{-N-\nu}$ is smallest. Therefore, when $|v|>R$, the extremal case is to take $f(v') \equiv m$ for all $v' \in \Xi \cap B_r$, for $r = c \left( \frac {M_0 |v|} {m} \right)^{1/N}$ and $f = 0$ elsewhere. Thus, in this case we get
%\[ \int_\Xi (m-f(v')) |v-v'| dv' \geq c m^{1+\nu/N} |v|^{-\nu/N-1}.\]
%Thus,
%\[ Q_1(f,f)(v) \leq -c m^{1+\nu/N} |v|^{\nu (1-1/N) + \gamma}  \ \text{ for } |v|>R.\]

\noindent \textbf{$\circ$ Lower bound for $Q_1$ for large values of $|v|$ using $E_0$.}

For $|v|>R$, the opening of $\Xi$ is bounded below by $c/|v|$ (i.e. the value of $\mu$ to be used in Lemma \ref{l:quantifyingintegralatmax} will be $c/|v|$). In order to estimate the right hand side of \eqref{e:Linfty-nc} for $|v|>R$, we will argue similarly to the case $|v|\leq R$, but we use a better estimate for $\int_{\Xi} f(v') \dd v'$. Indeed, since $|v'| > |v|/2$ for all $v' \in \Xi$,

\begin{equation} \label{e:ub2}  \int_{\Xi} f(v') \dd v' \leq 4E / |v|^2 \leq 4E_0 / |v|^2.
\end{equation}
We estimate the right hand side in \eqref{e:Linfty-nc} for $|v|>R$ using \eqref{e:ub2}. Applying Lemma \ref{l:quantifyingintegralatmax}, we get
\[ \int_\Xi (m-f(v')) |v-v'|^{-N-\nu} dv' \geq c m^{1+\nu/N} |v|^{\nu/N-1}.\]
Thus,
\begin{equation} \label{e:ubq13}  Q_1(f,f)(v) \leq -c 
m^{1+\nu/N} |v|^{\nu (1+1/N) + \gamma}  \ \text{ for } |v|>R.
\end{equation}

\noindent \textbf{$\circ$ Lower bound for $Q_1$ for large values of $|v|$ using $\int (1+|v|)^q f^p \leq K_0$.}

For $|v|>R$, the opening of $\Xi$ is bounded below by $c/|v|$. In order to estimate the right hand side of \eqref{e:Linfty-nc} for $|v|>R$, we will argue similarly as before but using the estimate on $\int (1+|v|)^q f^p$ instead of $E_0$. Since $|v'| > |v|/2$ for all $v' \in \Xi$,

\begin{equation} \label{e:ub3}  \int_\Xi |f(v')|^p \dd v' \leq \frac 1 {|v|^q} \left( \int_{\R^N} (1+|w|)^q |f(w)|^p \dd w \right)
\end{equation}
Applying now Lemma \ref{l:quantifyingintegralatmax}, we get
\[ \int_\Xi (m-f(v')) |v-v'| dv' \geq c m^{1+p\nu/N} |v|^{(q-1)\nu/N-1}.\]
Thus,
\[
 Q_1(f,f)(v) \leq -c 
m^{1+p\nu/N} |v|^{\nu (1+(q-1)/N) + \gamma}  \ \text{ for } |v|>R.
\]

Since we chose $q = \max(0,1-\frac N \nu (\nu+\gamma))$, then $\nu (1+(q-1)/N) + \gamma \geq 0$ and
\begin{equation} \label{e:ubq14}  Q_1(f,f)(v) \leq -c 
m^{1+p\nu/N} \ \text{ for } |v|>R.
\end{equation}

\noindent \textbf{$\circ$ Summary and contradiction for $\gamma+\nu > 0$.}

For any value of $v \in \R^N$, we combine \eqref{e:ubq11} with \eqref{e:ubq13} and get
\begin{equation} \label{e:ub-Q1}  Q_1(f,f)(v) \leq -c 
m^{1+\nu/N} (1+|v|)^{\nu (1+1/N) + \gamma}.
\end{equation}

% or $r = c \left( \frac {M_0} {m} \right)^{1/N}$ for $|v| \leq R$, and $f = 0$ elsewhere. 

%Note that we may be losing some information in the previous step. The optimal choice we described above is the value of $f$ which makes the right hand side in \eqref{e:Linfty-nc} the least negative possible, given that $m \geq f$ and the integral of $f$ is controlled. This \emph{optimal} $f$ does not satisfy all the assumptions in this theorem, but it suffices to get a rough bound for \eqref{e:Linfty-nc}.

Let us first analyse the case $\gamma+\nu>0$, using \eqref{e:ub-Q2} and \eqref{e:ub-Q1}, we get
\[ f_t = Q(f,f) \leq -c 
m^{1+\nu/N} (1+|v|)^{\nu (1+1/N) + \gamma} + Cm^{1+\gamma^-/N} (1+|v|)^{\gamma^+}. \]
Here, where we write $\gamma^-$, we mean $(|\gamma|-\gamma)/2$ and $\gamma = \gamma^+ - \gamma^-$. This inequality holds for $t>t_0$ and for $v$ being the point where the maximum $m$ is achieved (or one of them if there are many).

Since in this case $\nu+\gamma > 0$, then in particular $\nu (1+1/N) + \gamma > \gamma^+ \geq 0$. Then $(1+|v|)^{\nu(1+1/N)+\gamma} \geq (1+|v|)^{\gamma^+} \geq 1$.

Since we assume $\nu+\gamma > 0$, then  $1+\nu/N > 1+\gamma^-/N$. We let $a= (2C/c)^{N/(\nu-\gamma^-)}$, where $c$ and $C$ are the constants in the equation displayed above. Recalling that $f(t,x,v) = m(t) > a$, we get
\[ 
f_t(t,x,v) \leq -\frac c 2 m(t)^{1+\nu/N} \text{ if } m(t)>a.
\]
Choosing $b$ such that $N/\nu b^{\nu/N} = c/2$, we found a contradiction with \eqref{e:mp2} in the case $\gamma+\nu > 0$.

\noindent \textbf{$\circ$ Summary and contradiction for $\gamma+\nu \leq 0$ assuming $\int (1+|v|)^q f^p \leq K_0$ for some $p > N/(N+\nu+\gamma)$.}

In the case $\gamma+\nu \leq 0$, we assume \eqref{e:extra-integrability} in order to get the contradiction. Note that we can assume without loss of generality that $p < N/(N+\gamma)$, since, otherwise, we can estimate all norms $\|f\|_{L^{p_1}}$ for $p_1 \in (1,p)$ using $\|f\|_{L^p}$ and $M_0$.

Combining \eqref{e:ubq12} with \eqref{e:ubq14}, for all values of $v \in \R^N$,
\begin{equation} \label{e:ub-Q1-Lp}  Q_1(f,f)(v) \leq -c 
m^{1+p\nu/N}.
\end{equation}

We use \eqref{e:ub-Q2-withLp} instead of \eqref{e:ub-Q2} and obtain
\begin{equation} \label{e:ub-ode2}  f_t(t,x,v) \leq -c 
m^{1+p\nu/N} + Cm^{2- p \frac{N+\gamma}N}. 
\end{equation}
We arrive to the same conclusion provided that $1+p\nu/N > 2 - p \frac{N+\gamma}N$. That is, when $p > N / (N+\nu+\gamma)$. We argue as before choosing $a$ such that
\[ 
f_t(t,x,v) \leq \begin{cases}
 -\frac c 2 m^{1+p\nu/N} & \text{ if } m>a,\\
 0  & \text{ if } m \leq a.
 \end{cases}
\]
and the contradiction follows as before, now for $\gamma+\nu \leq 0$.
\end{proof}

%\begin{remark}
%Note that some steps in the proof above are not optimal, especially in the case $\gamma+\nu \leq 0$. There are several bounds that lead to the proof of Theorem \ref{t:Linfty}, but there is no function $f$ matching the worst case scenario of all of them. The function $f$ which corresponds to the worst case scenario of \eqref{e:Linfty-nc} is constant in a ball centered around $v$, and zero outside. This is the same worst case scenario for the bound on $Q_2$ \eqref{e:ub-Q2}. However, in this case the lower bound $\lambda$ in the kernel $K_f$ would be very large when $\gamma+\nu<0$ according to \eqref{e:Kg-approx}. Finding sharp estimates for $Q(f,f)$ as a whole may lead to an improvement of this $L^\infty$ estimate in the range $\gamma+\nu \leq 0$.
%\end{remark}

\begin{remark}
In \cite{gamba2009upper}, a Maxwellian upper bound for the homogeneous Boltzmann equation in the cut-off case is obtained using a maximum principle idea. They prove that if $g$ is a function which satisfies the linear equation
\[ g_t \geq Q(f,g),\]
and $g(0,v) \geq f(0,v)$ for all $v \in \R^N$, then $g \geq f$ for all times. One could speculate with proving Theorem \ref{t:Linfty} using $g(t,v) = C(t)$ (constant with respect to $v$). This approach, by itself, does not work. Indeed, for such function $g$, we would have $Q_1(f,g)=0$ and $Q_2(f,g)>0$. It is the negativity of the term $Q_1(f,f)$ what allows us to succeed in the proof, and this follows from a non linear analysis of the term $Q_1(f,f)$.
\end{remark}

\begin{remark}
The periodicity assumption with respect to the space variable $x$ is made simply for the convenience of saying that the maximum of $f(t,x,v)$ is achieved for every fixed value of $t$. For any appropriate decay condition which assures this, the proof would work without modification.
\end{remark}

\begin{remark} \label{r:enhanceddiffusion}
A simplified version of the proof of Theorem \ref{t:Linfty} would prove the following fact. Let $\gamma \in (-N,0]$ and $s \in (0,1)$. If $f$ is a function satisfying the inequality
\[ f_t(t,v) + (-\Delta)^s f(t,v) \leq f(t,v) \,\left(  \int |w|^\gamma \, f(t,v-w) \dd w \right),\]
and moreover, we assume that
\[ \sup_{t \in [0,\infty)} \|f(t,\cdot)\|_{L^p} \leq K_0 \text{ for some } p > N/ (N+2s+\gamma) \]
then $\|f(t,\cdot)\|_{L^\infty} \leq C(t)$ for some function $C(t)$ depending on $t$, $p$, $K_0$ and dimension.

In this case the result seems to be optimal, in the sense that if $p < N/ (N+2s+\gamma)$ there would be no upper bound in general for $\|f(t,\cdot)\|_{L^\infty}$.

It is not clear if the condition \eqref{e:extra-integrability} is really necessary for the result of Theorem \ref{t:Linfty} to hold. When $\gamma+\nu<0$, the kernel $K_f$ in the expression for $Q_1$ of Corollary \ref{c:Kf-approx} tends to be larger around the points where $f$ is larger. This \emph{enhanced diffusion} where $f$ is larger is not taken advantage in the proof of Theorem \ref{t:Linfty}, or in the majority of the estimates for the non cut-off Boltzmann equation so far.  Indeed, the intuition driving the coercivity inequalities and also, to some extent, the methods in this paper is that
\[ -Q(f,f) = c(-\Delta)^{\nu/2} f + \text{(lower order terms)}.\]
This is only an intuitive statement which, strictly speaking, is false. A more accurate statement would be
\[ -Q(f,f) = c(-\Delta)^{\nu/2} f + \text{(a monotone operator)}+ \text{(lower order terms)}.\]

Regularity results for the Boltzmann equation will always be conditional for $\gamma+\nu < 0$ unless the positive definite term in the middle can be taken advantage of.
\end{remark}

\section{H\"older bounds}

In this section we prove Theorem \ref{t:calpha-linearized-intro}, and corollary \ref{c:Calpha-Boltzmann-intro}.

We start with the $C^\alpha$ estimate for the linear Boltzmann equation. This result is practically a raw application of Theorem \ref{t:Calpha-general} after the analysis of the terms $Q_1$ and $Q_2$ in the previous sections.

\begin{thm} \label{t:Calpha-linearized}
Assume $B$ has the form \eqref{e:B-assumption}. Let $f:[0,T] \times \R^N \to \R$ be a non negative function such that
\begin{align*}
\int_{\R^N} f(t,v) \log f(0,v) \dd v &\leq H_0, \\
\int_{\R^N} |v|^2 f(t,v) \dd v &\leq E_0, \\
M_1 \leq \int_{\R^N} f(t,v) \dd v &\leq M_0, \\
\sup_{(t,v) \in [0,T]\times B_R} \int_{\R^N} |w|^\gamma f(t,v+w) \dd v &\leq K_0 && \text{\bf (only necessary if $\gamma < 0$)}, \\
\sup_{(t,v) \in [0,T]\times B_R} \int_{\R^N} |w|^{\gamma+\nu} f(t,v+w) \dd v &\leq \tilde K_0 && \text{\bf (only necessary if $\gamma+\nu > 2$)}
\end{align*}
Let $R>0$, $h \in [0,T] \times B_R \to \R$ and $g : [0,T] \times \R^N \to \R$ be a solution of the equation
\[ g_t = Q(f,g) + h \text{ in } [0,T] \times B_R.\]
Then, the following a priori estimate holds. For some $\alpha>0$ and $C>0$,
\[ \|g\|_{C^\alpha([T/2,T] \times B_{R/2}} \leq C \left( \|g\|_{L^\infty([0,T]\times \R^N} + \|h\|_{L^\infty([0,T]\times B_R} \right).\]
The constants $\alpha$ and $C$ depend on the cross section $B$, the dimension $N$, $R$, $T$ and the values of $M_0$, $M_1$, $H_0$, $E_0$, $K_0$ and $\tilde K_0$ only.
\end{thm}

\begin{remark}
The bound involving $K_0$ is redundant if $\gamma \geq 0$ because it can be controlled by $M_0$ and $\tilde K_0$. In the case of $\gamma < 0$, it is a requirement on higher integrability on $f$. The value of $\tilde K_0$ is redundant if $\gamma + \nu < 2$ because it can be controlled with $K_0$ and $E_0$. If $\gamma+\nu>2$, the value of $\tilde K_0$ is a bound on higher moments of $f$.
\end{remark}

\begin{remark}
For this result, it is not necessary that $f$ solves the Boltzmann equation \eqref{e:space-homogeneous-boltzmann}. Naturally, when it does, then that helps us control the various quantities involved in the assumptions.

The natural applications of this result are for $f$ being the solution of \eqref{e:space-homogeneous-boltzmann} and either $g=f$ or $g$ equal some derivative of $f$. The right hand side $h$ can be used potentially to handle lower order terms.
\end{remark}

\begin{proof}
We split the operator $Q(f,g)$ into two terms $Q_1(f,g)$ and $Q_2(f,g)$ as described in section \ref{s:twoterms}.
\[ g_t = Q_1(f,g) + Q_2(f,g) + h. \]
Applying Lemmas \ref{l:expression-for-Q1} and \ref{l:expression-for-Q2},
\begin{equation} \label{e:ca1}  g_t - \left( \int_{\R^N} (g(t,v')-g(t,v)) K_f(t,v,v') \dd v' \right) = (\tilde B \ast f) g + h. 
\end{equation}
Recall that from Lemma \ref{l:Btilde}, $\tilde B(v) = C |v|^\gamma$. Therefore
\[ \|\tilde B \ast f\|_{L^\infty([0,T] \times B_R)} \leq C K_0.\]
Thus, we bound the $L^\infty$ norm of the right hand side of \eqref{e:ca1}.
\[ \norm{ (\tilde B \ast f) g + h }_{L^\infty([0,T] \times B_R)} \leq C K_0 \|g\|_{L^\infty([0,T] \times B_R)} + \|h\|_{L^\infty([0,T] \times B_R)}.\]
We apply Lemmas \ref{l:K-above} and \ref{l:K-below} for every fixed value of $t$. From Lemma \ref{l:K-above} we obtain that $K_f$ satisfies \eqref{e:K-bound-above}. The constant $\Lambda$ depends only on $\tilde K_0$ and $B$. From Lemma \ref{l:K-below}, we obtain the existence of the set $A=A(t,v)$ so that $K_f$ satisfies \eqref{e:K-bound-below}. The values of $\lambda$ and $\mu$ depend only on the quantities $M_0$, $M_1$, $E_0$ and $H_0$ above. Applying Theorem \ref{t:Calpha-general}, we finish the proof.
\end{proof}

\begin{cor} \label{c:Calpha-Boltzmann}
Assume $B$ has the form \eqref{e:B-assumption}. Let $R>0$ and $f:[0,T] \times \R^N \to \R$ be a non negative function which solves \eqref{e:space-homogeneous-boltzmann}. Assume that
\begin{align*}
M_1 \leq \int_{\R^N} f(0,v) \dd v &\leq M_0, \\
\int_{\R^N} |v|^2 f(0,v) \dd v &\leq E_0, \\
\int_{\R^N} f(0,v) \log f(0,x,v) \dd v &\leq H_0, \\
%\sup_{(t,v) \in [0,T] \times B_R} \int_{\R^N} |w|^{\gamma+\nu} f(t,v+w) \dd v &\leq K_0 && \text{\bf (only necessary if $\gamma+\nu<0$)}, \\
\sup_{t \in [0,T]} \int_{\R^N} (1+|v|)^q f(t,v)^p \dd v &\leq K_0, \\
\text{ for some  } p &> N / (N+\nu+\gamma) \text{ and } q = \max(0,1-\frac N \nu (\nu+\gamma)). \qquad \text{\bf (only necessary if $\gamma+\nu \leq 0$)}
\end{align*}
Then, there is an a priori estimate of the form
\[ \|f\|_{L^\infty([T/2,T]\times \R^N)} + \|f\|_{C^\alpha([T/2,T] \times B_{R/2})} \leq C.\]
where $\alpha$ and $C$ depend only on the quantities $M_0$, $M_1$, $E_0$, $H_0$, $K_0$, $R$, $T$, the dimension $N$ and the cross section $B$.
\end{cor}

\begin{remark}
Note that when $\gamma + \nu \geq 0$, the assumptions are only the boundedness of the macroscopic quantities associated with the initial condition. The last condition is redundant since we always have it with $p=1$ and $q < 1$ and $K_0$ depending on $M_0$ and $E_0$. When $\gamma+\nu < 0$ we have the undesirable extra hypothesis of the boundedness of the value of $K_0$, which does not correspond to any macroscopic quantity.
\end{remark}

\begin{proof}
First, note that if $\gamma>2$, then all moments of $f$ are bounded for $t>0$. This is well known and can be found as Theorem 1 in chapter 2, section 2, of \cite{villani2002review}. Therefore, for any values of $\gamma$ and $\nu$, we can apply Theorem \ref{t:Linfty} to obtain that $\|v(t,\cdot)\|_{L^\infty(\R^N)} \in C(t)$ for all $t>0$.

In the case $\gamma<0$, since $f(t,\cdot) \in L^\infty \cap L^1$, then for all $t > 0$ and $v \in \R^N$,
\[ \int_{\R^N} |w|^\gamma f(t,v+w) \dd w \leq C.\]
This constant $C$ depends only on the parameters in the hypothesis of the Theorem. However, the inequality only holds for strictly positive time.

We restart the equation at a slightly positive time and apply Theorem \ref{t:Calpha-linearized} with $g=f$ to finish the proof.
\end{proof}

\appendix

\section{The Jacobian of the change of variables $(v_\star,\sigma) \to (v',w)$}

The change of variables needed to obtain the expression of $K$ in \eqref{e:Q1} is not so obvious because $\sigma$ and $w$ are variables living on $N-1$ dimensional surfaces. We work out the change of variables in this section.

\begin{lemma}  \label{l:change-of-variables}
For any non negative function $F$ (in terms of $v_*$, $v$, $r=|v-v_\ast|$, $\theta$, $v'$ and/or $v'_\ast$)
\[
 \int_{\R^N} \int_{S^{N-1}} F \dd \sigma \dd v_\star = 2^{N-1} \int_{\R^N} \frac{1}{|v'-v|} \int_{\{w : w \cdot v' = 0\}} F \ \frac 1 {r^{N-2}} \dd w \dd v'.\]
In the right hand side, we must write $v_\ast = v'+w$ and write the values of $r$, $\theta$ and $v'_\ast$ accordingly.
\end{lemma}

\begin{proof}
We start from the integral
\[ \int_{\R^N} \int_{S^{N-1}} F \dd \sigma \dd v_\star.\]

We will rewrite it as an integral in terms of $\dd w \dd v'$. 

Without loss of generality, we assume that $F$ is continuous. Otherwise, we can always approximate appropriately integrable functions $F$ by continuous functions. Moreover, we extend $F$ (continously) to values of $\sigma \in \R^N$. In that way, we can approximate the integral with
\[ \lim_{\eps \to 0} \int_{\R^N} \int_{\R^N} F \, \varphi_\eps(|\sigma|) \dd \sigma \dd v_\star.\]
where
\[ \varphi_\eps(t) = \begin{cases}
\eps^{-1} & \text{if } 0 \leq t \leq 1+\eps, \\
0 & \text{otherwise.}
\end{cases}
\]

Now we recall the formulas for $v'$ and $w$. In order to have cleaner expressions, we assume $v = 0$. Thus, $r = |v_\star|$ and
\begin{align*}
v' &= \frac 12 (v_\star + |v_\star| \sigma), \\
w &= \frac 12 (v_\star - |v_\star| \sigma).
\end{align*}
Since now both $(v_\star,\sigma)$ and $(v',w)$ are variables in $\R^{N+N}$, we compute the Jacobian of the change of variables
\[ \dd w \dd v' = \frac 1 {2^{2N}} 
\det \begin{pmatrix}
I + \sigma \otimes \frac {v_\star}{|v_\star|} & |v_\star| I \\
I - \sigma \otimes \frac {v_\star}{|v_\star|} & -|v_\star| I
\end{pmatrix} \dd \sigma \dd v_\star
= \frac{r^{N}}{2^{N}}  \dd \sigma \dd v_\star
\]

We have then obtained
\[ \int_{\R^N} \int_{S^{N-1}} F \dd \sigma \dd v_\star = \lim_{\eps \to 0} \int_{\R^N} \int_{\R^N} F \varphi_\eps \left( \frac {|v'-w|}{|v'+w|} \right) \frac{2^{N}}{r^N} \dd w \dd v' \]

Naturally, as $\eps \to 0$, the factor $\varphi_\eps(\dots)$ will concentrate along the plane $w \cdot v' = 0$. The function $\varphi_\eps$ equals $\eps^{-1}$ on its support. 

Let $w = w_p + w_b$, where $w_p \cdot v' = 0$. The function $\varphi_\eps(|w-v'|/|w+v'|)$ will be equal to $\eps^{-1}$ provided that
\[ 0 > 2 \frac{|v'| |w_b|}{r^2} > - \eps - o(\eps) .\]

That is, the width of the support of $\varphi_\eps(\dots)$ around the point $w_p$ is approximately $\eps r^2 / (2|v'|)$. Therefore, by taking $\eps \to 0$, we obtain
\begin{align*}
 \int_{\R^N} \int_{S^{N-1}} F \dd \sigma \dd v_\star &= \lim_{\eps \to 0} \int_{\R^N} \int_{\R^N} F \varphi_\eps \left( \frac {|v'+w|}{|v'-w|} \right) \frac{2^{N}}{r^N} \dd w \dd \sigma, \\
 &= \int_{\R^N} \int_{\{w : w \cdot v' = 0\}} F \ \frac {r^2}{2|v'|} \ \frac{2^{N}}{r^N} \dd w \dd v', \\
 &= 2^{N-1} \int_{\R^N} \frac 1{|v'|} \int_{\{w : w \cdot v' = 0\}} F \ \frac 1 {r^{N-2}} \dd w \dd v'.
\end{align*}
\end{proof}

The following Proposition is an elementary change of variables. We state it without proof. The first part is simply the usual Jacobian of polar coordinates. The second part is less common but equally elementary. 

\begin{prop} \label{p:dual-polar}
Let $g$ be any non negative function. The following two identities hold.
\begin{align*}
\int_0^\infty \int_{\partial B_1} g(r\sigma) \dd \sigma \dd r =  \int_{\R^N} g(z) \frac{\dd z}{|z|^{N-1}},\\
\int_{\partial B_1} \int_{\{w : w \cdot \sigma = 0\}} g(w) \dd S(w) \dd \sigma =  c_N \int_{\R^N} g(z) \frac{\dd z}{|z|}.
\end{align*}
Here $\dd S(w)$ stands for the differential of area of the variable $w$ on the hyperplane $\{w : w \cdot \sigma = 0\}$, and $c_N$ is a constant depending on dimension.
\end{prop}
 
\bibliographystyle{plain}
\bibliography{boltzmann}
\index{Bibliography@\emph{Bibliography}}%

\end{document}